\documentclass[12pt,reqno]{amsart}%
\usepackage{amsfonts}
\usepackage{geometry}
\usepackage{amsmath}
\usepackage{amssymb}
\usepackage{graphicx}
\usepackage{hyperref}%
\providecommand{\U}[1]{\protect\rule{.1in}{.1in}}
\hypersetup{colorlinks,linkcolor={red},citecolor={blue},urlcolor={green}}
\newtheorem{theorem}{Theorem}[section]
\theoremstyle{plain}
\newtheorem{acknowledgement}{Acknowledgement}[section]

\newtheorem{definition}{Definition}[section]

\newtheorem{lemma}{Lemma}[section]

\newtheorem{proposition}{Proposition}[section]
\newtheorem{remark}{Remark}[section]

\newcommand{\RR}{\mathbb{R}}
\numberwithin{equation}{section}

\allowdisplaybreaks

\date{\today}

\begin{document}
	\title[]{Decay estimates for time-fractional porous medium flow with nonlocal pressure}
	\author{Nguyen Anh Dao, Vu Tien Anh Nguyen}
	\address{N. A. Dao: School of Economic Mathematics and Statistics, University of Economics Ho Chi Minh City (UEH), Ho Chi Minh City, Viet Nam
	}\email{anhdn@ueh.edu.vn}

	\address{V. T. A. Nguyen: Faculty of Mathematics and Computer Science, University of Science, Ho Chi Minh City, Vietnam}
	\email{nguyenvutienanhmath1111@gmail.com}
	\subjclass[2010]{35R11, 45K05, 47G20, 35A23}
	\keywords{Time-fractional derivative,  Nonlocal pressure, Fractional Sobolev spaces, Stroock--Varopoulos inequality,  Decay estimates}
	\begin{abstract}
		The main purpose of this paper is to study weak solutions of a time-fractional porous medium equation with nonlocal pressure:
		\[ 
		\partial^\alpha_t u=\operatorname{div}\left( |u|^{m}\nabla (-\Delta)^{-s} u\right)   \,\, \text{in }  \mathbb{R}^N\times (0,T) \,,
		\]
		with $m\geq 1$, $N\geq 2$, $\frac{1}{2}\leq s<1$, and $\alpha\in(0,1)$.
		\\
		We first prove the existence of weak solutions to the equation with initial data in $L^1(\RR^N)\cap L^\infty(\RR^N)$  (possibly mixed sign). After that, we establish the $L^q-L^\infty$ decay estimate of weak solutions:
		\[ \|u(t)\|_{L^\infty(\RR^N)} \leq C t^{-\frac{\alpha}{q(1-\lambda_0)+ m}} \|u_0\|_{L^{q}(\RR^N)}^{\frac{q(1-\lambda_0)}{q(1-\lambda_0) + m}} ,\quad \text{for   }   t\in(0,\infty), 
		\]
		with $\lambda_0=\frac{N-2(1-s)}{N}$.
	\end{abstract}
	\maketitle
	\tableofcontents
	\section{Introduction}
	In this paper, we would like to study weak solutions of  the following  equation
	\begin{align}\label{1}
		\left\{
		\begin{array}
			{ll}%
			\frac{d}{dt} \left[k* (u-u_0)\right]	 	=\operatorname{div}\left( |u|^{m}  	\nabla  (-\Delta)^{-s}u  \right)   \,\, &\text{in }  \mathbb{R}^N\times (0,T)  \,,
			\\
			u(x,0)=u_0(x) &\text{in}
			~~\mathbb{R}^N  \,,
		\end{array}
		\right.
	\end{align}
	where $m\geq 1$, $\frac{1}{2}\leq s<1$, and space dimension $N \geq 2$. The symbol $(-\Delta)^{-s}$ denotes by the inverse of the fractional Laplacian operator, and kernel $k\in L^1_{{\rm loc}}(\RR_+)$ is of  $\mathcal{PC}$  type, see Definition \ref{Def2.1}. Through the paper, we denote 
	\[k * u(t) =\displaystyle\int^t_0 k(t-\tau)  u(\tau)\, d\tau\,.\] A typical kernel for our study is $k(t) =  g_{1-\alpha} (t) e^{-\nu t}$, $\nu\geq 0$, where $g_\alpha(t)=\frac{t^{\alpha-1}}{\Gamma(\alpha)}$  (the so-called  Riemann--Liouville kernel), and   $\Gamma(\beta)=\int_{0}^{\infty}  \lambda^{\beta-1}e^{-\lambda} \,d\lambda$,  $\beta>0$ is the usual gamma function. In particular, if we take $\nu=0$, then $\frac{d}{dt}\left[ k* (u-u_0)\right]$  becomes the time-fractional derivative of Caputo-type. In this case, we denote it by $\partial^\alpha_{t}  u(t)$ for short. 
	\\
	For a definitive study, we would like to consider $k(t)=g_{1-\alpha}(t)$ in this paper, even our proofs below also hold for $k(t)=g_{1-\alpha}(t)e^{-\nu t}$, $\nu\geq 0$. Moreover, we always assume that initial data  $u_0\in L^1(\RR^N)\cap L^\infty(\RR^N)$ except Theorem \ref{MainThe3}. 
	\\ 
	When  $\alpha=1$, and $s=0$,  Eq \eqref{1}  is a porous medium equation. This model arises from considering a compressible fluid with a nonnegative density distribution $u(x,t)$, and with Darcy’s law leading to the equation
	\begin{equation}\label{1.2}
		u_t - \operatorname{div} (u  \vec{\mu} ) = 0\,,
	\end{equation}
	where  $\vec{\mu} =\nabla {\bf p}$ is the potential, and  ${\bf p}$ is the pressure. Many other different relations between the density,
	the potential, and the pressure arising in the applications. For example, the model
	proposed by Leibenzon and Muskat states a law in which ${\bf p}= h(u)$, where $h$ is a
	nondecreasing scalar function (see more examples in \cite{vazB1}). In the paper, we study a potential that takes into account long-range interactions, namely ${\bf p} = (-\Delta)^{-s} u$.
	Eq \eqref{1.2} in terms of nonlocal pressure has been studied by many authors in \cite{BiImKa, Ca-So-va, Ca-va, Ca-va2011, Anh1, StTsvz17} and the references cited therein. There are many questions, addressed to equations of this type, which are the object of active research, such as those pertaining to existence and uniqueness, regularity, the behavior of solutions in short time and in large time, and so on.
	\\
	
	In \cite{Cap1999}, Caputo introduced another model of the porous medium equation
	\begin{equation}\label{1.3}
		\partial^\alpha_t u-\operatorname{div}
		\left( \kappa(u)  \nabla  u \right) = f 
	\end{equation} 
	by modifying  Darcy's law.  Eq \eqref{1.3} is a  time-fractional diffusion equation. A strong motivation for investigating such equation comes from physics and engineering for memory effect, viscoelasticity, porous media, etc  (see, e.g.,  \cite{Cap1999, Co-Ca, Go-Ma-Mo-Pa, Ku-Ry-Ya,  Mai1996, Mai1, Mai2, Met-Kla} and the references therein).
	\\
	At the moment, we want to present several mathematical works for the time-fractional equations in order to motivate our research. In \cite{Zacher2009}, Zacher proved the existence and uniqueness of weak solutions to certain abstract evolutionary integro-differential equations in Hilbert spaces
	\begin{equation}\label{1.4}
		\left\{\begin{array}{cl}
			&\frac{d}{dt}  \left< k*\left(u(t)-u_0\right),v\right>_{\mathcal{H}} + a\left(t, u(t), v\right)  =\left<f(t), v\right>_{\mathcal{V}',\mathcal{V}}  ,  \quad v\in \mathcal{V}, \, \text{a.e. }  t\in(0,T) \,,
			\\
			&u(0)  = u_0 \,, 
			\\
			&u\in H^{1}_2\left([0,T];\mathcal{V}'\right)\cap L^2\left( [0,T]; \mathcal{V}\right)\,,
		\end{array}\right.
	\end{equation}
	where  $a : (0,T) \times \mathcal{V} \times \mathcal{V} \to \RR$ is a
	coercive bilinear form; and  $\mathcal{V}, \mathcal{H}$ are real separable Hilbert spaces such that $\mathcal{V}$  is densely, and continuously embedded into $\mathcal{H}$, and 
	$\mathcal{V}  \hookrightarrow \mathcal{H}  =  \mathcal{H}'\hookrightarrow\mathcal{V}'$. We would like to mention that Allen \cite{Allen1} also obtained the unique result for a different class of weak solutions that correspond to the Marchaud derivative.
	\\
	
	Furthermore, the short-time and long-time behavior of solutions of the time-fractional diffusion equations of non-divergence and divergence form have been studied by the authors in \cite{Ke-Si-ve-Za, Ke-Si-Za,  Ko1990, Ko2008, Li3, Zac3, Xiao}, and the references cited therein. In particular, such behavior have been obtained for mild solutions, represented via Duhamel's formula. For a simple introduction and for our proof below, let us consider the following equation 
	\begin{equation}\label{1.5a}
		\left\{  \begin{array}{cl}
			\partial^\alpha_t u  = \Delta u +  f,\, &(x,t)\in\RR^N \times(0,T)  \,,  
			\\
			u(x,0)  =u_0(x) ,\, &x\in\RR^N  \,.
		\end{array} \right.
	\end{equation}
	It is known that if $u_0$ and $f$ are sufficiently regular, then Eq \eqref{1.5a} possesses a unique mild solution 
	\begin{align}\label{1.6}
		u(x,t)& =  \int_{\RR^N}  Z(x-\xi,t)  u_0(\xi)  \, d\xi +  \int^t_0  \int_{\RR^N}  Y(x-\xi,t-\tau)  f(\xi, \tau)  \, d\xi  d\tau   \nonumber
		\\
		&  :=     Z(t) * u_0  +\int^t_0  Y(t-\tau)  *f(\tau)  \, d\tau   \,,
	\end{align}
	with  
	\[ 
	\widehat{Z}(\xi,t)=   E_{\alpha,1}(- t^\alpha|\xi|^2)  ,\, \text{ and   } \, \widehat{Y}(\xi, t)  = t^{\alpha-1} E_{\alpha, \alpha}\left(- t^\alpha|\xi|^2\right)\,,
	\]
	where $E_{\alpha, \beta}(z) =\displaystyle \sum^{\infty}_{k=0}  \frac{z^k}{\Gamma (\alpha k+\beta)}$, $z\in\mathbb{C}$    is
	the Mittag--Leffler function.
	\\ 
	Concerning the pointwise estimates of fundamental solutions  $Z(x,t), Y(x,t)$, we refer to \cite[Chapter 5]{Ko2004-book} for details, see also  the seminal work by Kochubei \cite{Ko1990,Ko2008}, and by 
	Eidelman--Kochubei \cite{Ei-Ko}. By exploiting the pointwise estimates of the fundamental solutions, one obtains the $L^r$-decay estimates of $Z(x,t), Y(x,t)$  (see, e.g., \cite{Ke-Si-ve-Za, Ke-Si-Za, Li3}).
	\\
	Precisely, for  $t>0$  we have
	\begin{equation}\label{1.7}
		\| Z (t)\|_{L^r(\RR^N)}\lesssim t^{-\frac{\alpha N}{2}\left(1-\frac{1}{r}\right)} ,  \text{ if }  1\leq  r< \kappa_1(N) ,\, \kappa_1 (N)=\left\{ \begin{array}{cl}
			\frac{N}{N-2}  &\text{ if  }   N\geq 3\,,
			\\
			\infty & \text{ otherwise } \end{array}\right. .
	\end{equation}
	Note that if $N\leq  2$, then \eqref{1.7} holds true for all $r\in[1,\infty]$.
	And, 
	\begin{equation}\label{1.8}
		\|Z (t)\|_{L^{\kappa_1(N),\infty}(\RR^N)}\lesssim t^{-\alpha}\,, \text{ if } N\geq 3 \,,
	\end{equation}
	where $L^{q,\infty}$ is the weak-$L^q$ space.
	\\
	For  $Y(x,t)$, one  has  
	\begin{equation}\label{1.9}
		\|Y(t)\|_{L^r(\RR^N)} \lesssim t^{\alpha-1-\frac{\alpha 
				N}{2}\left(1-\frac{1}{r}\right) } \,,  \text{ if }   1\leq r <\kappa_2(N) \,, N>4 \, , 
	\end{equation}
	with \, $\kappa_2(N)= \left\{  \begin{array}{cl}
		&  \frac{N}{N-4} , \text{  if  }  N>4 ,
		\\
		& \infty , \text{  otherwise  }
	\end{array}\right.$. If $N<4$, then \eqref{1.9}  holds true for all $r \in [1,\infty]$.
	\\
	In addition, 
	\begin{equation}\label{1.10a}
		\| Y (t)\|_{L^{\kappa_2 (N),\infty} (\RR^N)}\lesssim t^{-\alpha-1} \,,  \text{ if }   N>4 \,.
	\end{equation}
	For our purpose later, we also present the $L^r$-estimate of $\nabla Y$. 
	\\
	Denote  $\kappa_3(N)=\frac{N}{N-3}$ if $N>3$,  and $\kappa_3(N)=\infty$  otherwise.  Then, the following statements hold true
	\begin{equation}\label{1.10c}
		\|\nabla Y(t)\|_{L^r(\RR^N)} \lesssim t^{\frac{\alpha}{2}-1-\frac{\alpha N}{2}\left(1-\frac{1}{r}\right)} \,,  \text{ if }   1\leq r <\kappa_3(N) , \,N>3 \,.
	\end{equation}
	If $N \leq 3$, then \eqref{1.10c} also holds for $r \in[1,\infty]$. 
	\\
	When $N>3$, and $r= \kappa_3(N)$, 
	\eqref{1.10c} only holds in the weak-$L^r$. That is 
	\begin{equation}\label{1.10d}
		\|\nabla Y(t)\|_{L^{r,\infty}(\RR^N)} \lesssim t^{-\alpha-1} \,.
	\end{equation}
	From $\eqref{1.7}-\eqref{1.10d}$, we can deduce the $L^q-L^p$ decay estimates of mild solutions of Eq \eqref{1.5a}. For instance, if $f=0$, then it follows from \eqref{1.7} and the Young inequality that
	\begin{equation}\label{1.10b}
		\|u(t)\|_{L^p(\RR^N)}= \|Z * u_0(t)\|_{L^p(\RR^N)} 
		\leq \|Z(t)\|_{L^{r}(\RR^N)} \|u_0\|_{L^{q}(\RR^N)}  \lesssim  t^{-\frac{\alpha N}{2}\left(1-\frac{1}{r}\right)} \|u_0\|_{L^{q}(\RR^N)} \,,
	\end{equation}
	provided that  $1\leq r<\kappa_1(N)$, $N\geq 3$,   and   $\displaystyle\frac{1}{p} +1  =\frac{1}{r}+\frac{1}{q}$.
	\\
	We also note that such decay estimates have been established by the authors in \cite{Ke-Si-Za} for the mild solutions to Eq \eqref{1.5a} type with fractional diffusion, and by the authors in \cite{Li3} for the mild solutions to Keller--Segel type time-space fractional diffusion equation.
	\\
	On the other hand, to obtain the above estimates for weak solutions of time-fractional diffusion equations, one uses the energy method. For example, the authors  \cite{Ke-Si-ve-Za} studied the decay estimates of weak solutions to the nonlocal in-time diffusion equation
	\begin{equation}\label{1.12}
		\left\{  \begin{array}{cl}
			&  \partial^\alpha_t u = \operatorname{div}  \left( A(x,t) \nabla u \right) , \quad (x,t)\in \RR^N\times (0,T)\,, 
			\\
			& u(x,0)  =u_0(x) ,\quad x\in\RR^N  \,,
		\end{array} \right.
	\end{equation}
	where $A \in L^\infty_{{\rm loc}} \big( \RR^N\times [0,\infty) ; \RR^{N\times N}\big)$ satisfies the  elliptic condition. Particularly, 
	the authors obtained  the $L^2$ decay estimate of weak solutions  to Eq \eqref{1.12}:
	\begin{equation}\label{1.13}
		\|u(t)\|_{L^2(\RR^N)} \leq C t^{-\frac{\alpha N}{N+4}} ,\quad t>0 \,, 
	\end{equation} 
	provided that $u_0\in L^1(\RR^N)\cap L^2(\RR^N)$.  Note that  
	constant $C>0$ in \eqref{1.13} depends on $\|u_0\|_{L^1(\RR^N)}$, $\|u_0\|_{L^2(\RR^N)}$, and the parameters involved,  see  \cite[Theorem 6.1]{Ke-Si-ve-Za}.
	\\
	Compare \eqref{1.13} to \eqref{1.10b} when  $p=r=2$, $q=1$,  one observes that the $L^2$ decay estimates of mild solution and weak solution are different. This point of view is very different from the fundamental heat equation with the Gaussian kernel  $G(x,t) = \displaystyle (4\pi t)^{-\frac{N}{2}} e^{-\frac{|x|^2}{4t}}$. 
	\\
	Besides, the energy method has been used by Vergara--Zacher \cite{Zac3} in order to obtain the decay estimates of weak solutions of Eq \eqref{1.12},  of time-fractional $p$-Laplace equation, and of the time-fractional porous medium equation with the homogeneous Dirichlet boundary condition.  In addition, this method is very useful in studying the regularity of weak solutions. For instance,  Zacher \cite{Zac1} obtained a H\"older continuous regularity of weak solutions of the time-fractional diffusion equations.    
	\\
	It is interesting to mention that Giga-Namba \cite{Giga} studied the Hamilton--Jacobi equations with Caputo’s
	time fractional derivative in  $\mathbb{T}^N:=  \RR^N \setminus \mathbb{Z}^N$. By introducing a suitable notion of viscosity solutions, the authors established the unique existence, stability, and regularity results of a viscosity solution.  
	\\
	
	Now, we discuss some recent results concerning our equation. When $m=1$,  and $0<s<\frac{1}{2}$,  Allen--Caffarelli--Vasseur  \cite{Caf2.0} obtained an existence of weak solutions to Eq \eqref{1} with initial data satisfying the exponential decay at infinity.  Moreover, the same authors \cite{Caf1.0} also established Holder's regularity of solutions to the equation with nonlocal integral.
	By following these papers, Djida--Nieto--Area \cite{Dji-Ni-Ar}  extended the existence result of solutions of Eq \eqref{1} to $m\geq 1$ with the same assumption of exponential decay at infinity on initial data. In addition, they proved the finite speed of propagation of solutions. That is solution is compactly supported for any time $t>0$  if the initial datum has compact support in $\RR^N$. We emphasize that the proofs of the existence of solutions to Eq \eqref{1} for $0<s<\frac{1}{2}$ in \cite{Caf2.0, Dji-Ni-Ar}  requires the exponential decay at infinity of initial data to construct the sequence of approximating solutions. 
	\\ 
	From our knowledge, Eq \eqref{1} has not been investigated for the case $\frac{1}{2}\leq s<1$, and initial data in $L^1(\RR^N)\cap L^\infty(\RR^N)$ possibly changed sign. Thus, we would like to study weak solutions of Eq \eqref{1} in this paper. Moreover, we establish the decay estimates of those solutions.
	\vspace{0.1in}
	\\
	\textbf{Notation.}
	Let us denote $X=L^1(\RR^N)\cap L^\infty(\RR^N)$,  equipped with the norm
	\[ 
	\|f\|_X  =  \|f\|_{L^1(\RR^N)} + \|f\|_{L^\infty(\RR^N)}  \,.
	\]
	Furthermore,  we set \,
	$\Theta(u) =  |u|^{m} \nabla (- \Delta )^{-s}u$,   $Q_T=\RR^N\times (0,T)$  for $ T>0$.
	And, $B_R$ is denoted by the open ball in $\RR^N$ with center at $0$, and radius $R$.
	\\
	Next, we denote $H^{-1}(\Omega)$ by the dual of $H^1_0(\Omega)$ for any bounded domain $\Omega$ in $\RR^N$.
	\\
	Through the paper, the constant $C$ may change step by step. Moreover, $C = C(\alpha, s, N, m)$
	means that  $C$ merely depends on the parameters $\alpha, s, N, m$.
	\\
	Finally, the notation
	$A \lesssim B$ means that there exists a universal constant $C > 0$ such that  $A \leq  C B$.
	\subsection*{Main results}  
	We introduce the notion of weak solutions to Eq \eqref{1}.  
	\begin{definition}\label{Def1}
		Let $u_{0} \in X$. We say that $u$ is a weak solution of Eq \eqref{1} if $u\in L^\infty\big(0,T; X\big)$, $\Theta(u)\in L^2(Q_T)$  for $T>0$, and for every test function $\varphi  \in H^1 \big(0,T; L^2(\RR^N)\big)  \cap   L^2\big(0,T; H^1(\RR^N)\big)$  with $\varphi(T)=0$ there holds
		\begin{align}\label{3.3}
			\int_{Q_T} \varphi_t  [g_{1-\alpha}*(u-u_0)]  \,  dxdt =  \int_{Q_T} \Theta(u) \cdot \nabla \varphi \, dxdt  \,.
		\end{align}
	\end{definition}
	Then, our first result is of the existence of weak solutions to Eq \eqref{1}.
	\begin{theorem}\label{MainThe1}  Let $\frac{1}{2}\leq s<1$,  $0<\alpha<1$,  $m\geq 1$,  and let $u_0\in X$. Then,  Eq \eqref{1} has a weak solution $u$ in $\RR^N\times(0,\infty)$. Moreover, for any $q\in [1, \infty]$, there holds true
		\begin{align}\label{1.10}
			\|u(t)\|_{L^q(\RR^N)} \leq \|u_0\|_{L^q (\RR^N)} \,,   \quad \text{for  } \, t\in (0,\infty)\,. 
		\end{align}
		Furthermore, for any $T>0$  we have
		\begin{equation}\label{1.11}  
			\int^T_0 \big\||u|^{\theta_q-1}u(t)\big\|^2_{H^{1-s}(\RR^N)} \, dt  \leq C(\alpha,m,q) T^{1-\alpha}\|u_0\|^q_{L^q(\RR^N)} \,, 
		\end{equation}
		with \, $\theta_q=\frac{m+q}{2}$, $q>1$.
	\end{theorem}
	\begin{remark}
		It follows from the regularity of the weak solution that   
		$u(t)$ converges to $u_0$ in $H^{-1}(B_R)$ for any $R>0$  (see Remark \ref{Rem4.2} below). Thus, $u(t)$ possesses an initial trace $u_0$ in this sense.
	\end{remark}
	Next, we point out the $L^q-L^\infty$ decay estimate of weak solutions to  Eq \eqref{1}.
	\begin{theorem}\label{MainThe2} Assume hypotheses as in Theorem \ref{MainThe1}. If $u$ is a weak solution of Eq \eqref{1} in $\RR^N\times(0,\infty)$, then for  $q\in(1,\infty)$   there exists a constant $C=C(\alpha,s,N, m, q)>0$ such that
		\begin{align}\label{1.5}
			\|u(t)\|_{L^\infty(\RR^N)} \leq C t^{-\frac{\alpha}{q(1-\lambda_0)+ m}} \|u_0\|_{L^{q}(\RR^N)}^{\frac{q(1-\lambda_0)}{q(1-\lambda_0) + m}} ,\quad \text{for   }   t\in(0,\infty),
		\end{align}
		with $\lambda_0=\frac{N-2(1-s)}{N}$.
	\end{theorem}
	\begin{remark}
		Such estimate \eqref{1.5} was obtained by the authors in \cite{BiImKa,Anh1,StTsvz17} when $\alpha=1$.  
	\end{remark}
	Our final result is of the $L^1-L^p$ decay estimates of very weak solutions. Here, $u$ is  called a very weak solution of Eq \eqref{1} if $u\in L^\infty\big(\RR^N \times (\tau, T)\big)$, $\Theta (u)  \in L^2\big(\RR^N \times (\tau, T) \big)$ for any $0<\tau<T<\infty$, and $u$ fulfills \eqref{3.3}  for all test functions $\varphi\in\mathcal{C}^\infty_c (Q_T)$ for $T>0$.  
	\begin{theorem}\label{MainThe3}  Let $\frac{1}{2}\leq s<1$,  $0<\alpha<1$,   $m\geq 1$, and let $u_0\in L^1(\RR^N) \cap L^{q_0}(\RR^N)$ for some $q_0\in (1, \infty)$. Then,  Eq \eqref{1} has a 
		very weak solution $u$ in $\RR^N\times(0,\infty)$. In addition,  $u$  satisfies \eqref{1.10} for $q\in[1,q_0]$,  and \eqref{1.11} for  $q\in (1,q_0]$. 
		\\
		Beside, for any $p\in (1,\infty]$  there exists a constant $C=C(\alpha,s,N, m, q_0, p)>0$ such that
		\begin{align}\label{1.7a}
			\| u(t) \|_{L^{p}(\RR^N)} \leq C t^{-\frac{\alpha\left(1-\frac{1}{p}\right)}{q_0(1-\lambda_0)+m}} \|u_0\|^{\frac{q_0(1-\lambda_0) \left(1-\frac{1}{p}\right)}{ q_0 (1-\lambda_0)+ m }}_{L^{q_0}(\RR^N)}
			\| u_0 \|^{\frac{1}{p}}_{L^{1}(\RR^N)}  ,\quad \text{for   }   t\in(0,\infty) \,.
		\end{align}
	\end{theorem}
	The paper is organized as follows. In the next section, we give some preliminary results involving the kernels of $\mathcal{PC}$ type, the fractional Sobolev spaces, and the fundamental inequalities of Riesz transforms, Riesz potential. In section 3, we study the regularizing equations to Eq \eqref{1}, and prove the a priori estimates to the approximating solutions. The proof of Theorem \ref{MainThe1} is given in Section 4.
	In the final section, we investigate the decay estimates of solutions through the proofs of  Theorems \ref{MainThe2}, \ref{MainThe3}. Moreover, we also establish another decay estimate via the comparison method.

	\section{Preliminary results}\label{Sec2}
	\subsection{Regularization of the kernel}
	We first define the kernels of $\mathcal{PC}$ type.
	\begin{definition}\label{Def2.1}
		A kernel $k \in L^1_{{\rm loc}}(\RR_+)$ is said to be of type $\mathcal{PC}$ if it is nonnegative and nonincreasing, and there exists a
		kernel $l \in L^1_{{\rm loc}}(\RR_+)$ such that $k * l = 1 $  in $(0,\infty)$. In this case, we say that $(k,l)$ is a $\mathcal{PC}$ pair and write $(k,l)\in\mathcal{PC}$.
	\end{definition} 
	An important example is given by
	\begin{equation}\label{2.0}
		k(t) = g_{1-\alpha} (t) e^{-\nu t},\,   \text{   and }  \,l ( t ) =  g_{\alpha} (t) e^{-\nu t} + \nu  \big( 1* [g_{\alpha} (\cdot) e^{-\nu \cdot}]\big)(t)  ,\quad t>0 \,,
	\end{equation}
	with $\alpha\in( 0 , 1 )$ and $\nu\geq  0$. Both kernels are strictly positive and decreasing. We refer to \cite{Ke-Si-ve-Za, Zac3} for the fundamental properties of these kernels.   
	In particular, if we take $\nu=0$  in \eqref{2.0}, then  $(k, l)=(g_{1-\alpha},g_\alpha)$. 
	\\
	Next, we introduce a smoothing effect to $g_{\alpha}$ by using the Yosida approximation.  
	Let $\mathcal{H}$ be a real Hilbert space. Then,  
	for $p\geq 1$  we define the operator $B$ by
	\begin{align*}
		B u=\frac{d}{d t}(g_{1-\alpha} * u), \quad D(B)=\left\{u \in L^p(0, T; \mathcal{H}): g_{1-\alpha} * u \in W_0^{1,p}(0, T ; \mathcal{H})\right\} \,,
	\end{align*}
	where the zero means the vanishing at $t=0$. 
	Its Yosida approximations $B_n$, defined by $B_n=n B(n+B)^{-1}$, $n \geq1$ enjoy the property that for any $u \in D(B)$, one has $B_n u \rightarrow B u$ in $L^p(0, T ; \mathcal{H})$ as $n \rightarrow \infty$. Furthermore, we have
	\begin{align}
		\label{2.1}
		B_n u=\frac{d}{d t}\left(g_{1-\alpha,n} * u\right), \quad u \in L^p(0, T ; \mathcal{H})  \,, 
	\end{align}
	where $g_{1-\alpha,n}=n s_{\alpha,n}$, and $s_{\alpha,n}$ is the solution of the scalar-valued Volterra equation
	\begin{align}
		\label{2.2}
		s_{\alpha,n}(t)+n\left(s_{\alpha,n}*g_\alpha\right)(t)=1, \quad t>0, \,n \geq 1.
	\end{align}
	Denote $h_{\alpha, n}$  by the resolvent kernel, associated with $n g_\alpha$, we have
	\begin{align}\label{2.3}
		h_{\alpha, n}\in L^1_{{\rm loc}}\left(\mathbb{R}_{+}\right),\quad	h_{\alpha,n}(t) +  n\left(h_{\alpha, n} * g_\alpha \right)(t)=n g_\alpha(t), \quad t>0,\, n \geq 1\,,
	\end{align}
	By taking the convolution  \eqref{2.3} with $g_{1-\alpha}$, and by using the fact $g_\alpha * g_{1-\alpha}=1$, we obtain
	\begin{align*}
		\left(g_{1-\alpha} * h_{\alpha, n}\right)(t)+n\left(\left[g_{1-\alpha} * h_{\alpha, n}\right] * g_\alpha\right)(t)=n\,, \quad t>0,
	\end{align*}
	which shows that
	\begin{align}
		\label{2.4}
		g_{1-\alpha, n}=n s_{\alpha, n}=g_{1-\alpha} * h_{\alpha, n}\,, \quad n\geq 1 .
	\end{align}
	In view of \eqref{2.4}, we conclude that the kernels $g_{1-\alpha, n} \in W^{1,1}(0, T)$ are nonnegative, and nonincreasing for all $n\geq 1$. 
	\\
	As a consequence, for any function $f \in L^p(0, T ; \mathcal{H})$, there holds true
	\begin{align}\label{2.5}
		h_{\alpha, n} * f \rightarrow f \text{ in } L^p(0, T ; \mathcal{H})\,,
	\end{align}
	as $n \rightarrow \infty$. In fact, defining $u=g_\alpha * f$, we have $u \in D(B)$, and
	\begin{align}\label{2.6a}		
		B_n u & =\frac{d}{d t}\left(g_{1-\alpha, n} * u\right)=\frac{d}{d t}\left(g_{1-\alpha} * g_\alpha * h_{\alpha, n} * f\right)\nonumber
		\\
		& =h_{\alpha, n} * f \rightarrow B u=f \quad \text { in } L^p(0, T ; \mathcal{H})\,,
	\end{align}
	as $n \rightarrow \infty$. In particular, $g_{1-\alpha, n} \rightarrow g_{1-\alpha}$ in $L^1(0, T)$ as $n \rightarrow \infty$.
	\\
	Now, we point out a fundamental identity for integro-differential operators of the form
	$\frac{d}{dt} ( k * u )$.
	\begin{lemma}\label{Lem-Iden1}
		Suppose $k\in W^{1,1} (0, T)$, and $G\in  C^1 (\RR)$. Then a straightforward computation shows that for a sufficiently smooth function $u$ on $( 0, T )$ one has for a.a.  $t \in ( 0, T )$ 
		\begin{align}\label{2.6}
			& G' \big(u ( t )\big) \frac{d}{dt}( k * u )( t ) =
			\frac{d}{dt} k * G ( u )  ( t )+
			\big[ -G \big(u ( t ) \big)
			+ G' \big(u ( t )\big)  u ( t ) \big] k ( t )\nonumber
			\\
			&+
			\int^t_0 \left\{
			G \big(u ( t- s ) \big)  -  G \big(u ( t ) \big) -  G'\big( u ( t )\big)  [u ( t-s )-u ( t ) ] \right\}\big(- \dot{k} ( s ) \big) ds \,. 
		\end{align}
	\end{lemma}
	\begin{remark}\label{Rem2.1}
		An integrated version of \eqref{2.6} can be found in  \cite[Lemma 18.4.1]{Gri1}.
		\\
		On the other hand, observe that the last term in \eqref{2.6} is nonnegative in case $G$ is convex and $k$ is nonincreasing. This fact will be used repeatedly below. 
	\end{remark}
	\subsection{Fractional Laplacian and Sobolev spaces}
	Let $p\geq 1$, and $s\in(0,1)$.   For a given  domain $\Omega\subset\mathbb{R}^N$,  we  define the fractional Sobolev space
	$$ W^{s, p} (\Omega) =\left\{   u\in L^p(\Omega): \int_\Omega \int_\Omega \frac{|u(x)-u(y)|^p}  {|x-y|^{N+sp}}  dxdy<\infty   \right\}, $$
	endowed with the  norm 
	\[
	\|u\|_{W^{s, p}(\Omega)}  = \left( \|u\|^p_{L^p(\Omega)} +  \int_\Omega \int_\Omega \frac{|u(x)-u(y)|^p}  {|x-y|^{N+sp}}  dxdy \right)^{1/p}.
	\]
	Moreover, we also denote the homogeneous   fractional Sobolev space by 
	$\dot{W}^{s, p} (\Omega)$, endowed with the seminorm   
	$$\|u\|_{\dot{W}^{s, p}(\Omega)}  =\left(\int_\Omega \int_\Omega \frac{|u(x)-u(y)|^p}  {|x-y|^{N+sp}}  dx dy \right)^{1/p}  .$$
	In  particular, we  denote  $W^{s,2}(\mathbb{R}^N)$ by  $H^s(\mathbb{R}^N)$, which turns out to be a Hilbert space. It is well-known that we have the equivalent characterization
	\[
	H^{s}(\mathbb{R}^N) = \left\{  u\in L^2(\mathbb{R}^N) : \int_{\RR^N}  (1+|\xi|^{2s}) |\mathcal{F}\{u\}(\xi)|^2  d\xi  <\infty \right\},
	\]
	where $\mathcal{F}$ denotes the Fourier transform, and we have
	\[
	\|u\|_{H^{s}(\mathbb{R}^N)}  = \left( \int_{\RR^N} (1+|\xi|^{2s}) |\mathcal{F}\{u\}(\xi)|^2  d\xi \right)^{1/2}.
	\]
	In addition, for $u\in H^{s}(\mathbb{R}^N)$, the  fractional Laplacian  is defined by 
	\begin{equation}\label{2.19}
		(-\Delta)^s u(x) 
		= C(N,s) P.V.  \int_{\RR^N}  \frac{u(x)-u(y)}{|x-y|^{N+2s}} dy 
		=\mathcal{F}^{-1} \{|\xi|^{2s} \mathcal{F}(u)(\xi) \}.
	\end{equation}
	Then, 
	$$ \|u\|^2_{H^{s}(\mathbb{R}^N)}=  \|u\|^2_{L^2(\mathbb{R}^N)}  +  C\|(-\Delta)^\frac{s}{2} u \|^2_{L^2(\mathbb{R}^N)} .$$
	Next, for $s>0$ we denote 
	\[ 
	(-\Delta)^{-s} f=\mathcal{I}_{2s}[f] =  \frac{1}{C(N,s)}\int_{\RR^N}  \frac{f(y)}{|x-y|^{N-2s}}  \, dy \,, 
	\]
	the Riesz potential, where $C(N,s)= \frac{\pi^{N/2} 2^{2s} \Gamma \left(s\right)}{\Gamma \left(\frac{N-2s}{2}\right)}$  (see \cite[Chapter 5]{Stein}). 
	Moreover,  the following inequality holds true
	\begin{equation}\label{2.18}
		\|\mathcal{I}_{s} (f) \|_{L^{q}(\RR^N)} \lesssim \|f\|_{L^{p}(\RR^N)}  
	\end{equation} 
	for all $f\in L^{p}(\RR^N)$,	with $1<p<\frac{N}{s}$, and $\frac{1}{q}= \frac{1}{p}-\frac{s}{N}$.
	\\
	In particular, if  $s=2$, then $\mathcal{I}_{2}$ is the inverse of the Laplacian operator.
	Note that if $N=s=2$,  then the above integral is a convolution of $f$ and  $ \frac{1}{2\pi} \log |x|$ instead of $C|x-y|^{s-N}$.
	\\
	In addition, it is known that  $\mathcal{R}(f)=\nabla \mathcal{I}_1(f)$  the Riesz transforms, and $\mathcal{R}$ maps $L^{p}(\RR^N)$ into $L^{p}(\RR^N)$ for $1<p<\infty$, see \cite[Chapter 3]{Stein}.
	\\
	Next, we recall a generalized version of the Stroock--Varopoulos inequality by Hung--V\'azquez  \cite{Hung}.  
	\begin{lemma}\label{LemSt-va} Let $\Omega$ be a bounded domain in $\RR^N$.  Let $s\in (0,1)$, and let $\psi, \phi\in \mathcal{C}^1(\mathbb{R})$ be such that $\psi', \phi'\geq 0$. Then,
		\begin{align}\label{2.20}
			\int_{\Omega} \psi(f) 
			(-\Delta)^s  \left[ \phi(f) \right] dx  \geq 0\,.
		\end{align}
		If we take $\psi(f)=f$, then we obtain
		\begin{align}\label{es73}
			\int_{\Omega} f (-\Delta)^s  \left[\phi(f)\right] \, dx \geq \int_{\Omega}  \left|(-\Delta)^\frac{s}{2} \Phi(f)\right|^2 \, dx  \,,
		\end{align}
		where $\phi'=(\Phi')^2$.
	\end{lemma}
	\begin{remark}\label{Rem2.3}
		A version of this lemma for $\RR^N$ was obtained by the authors in \cite{Anh1, StTsvz17}.
	\end{remark}
	Finally, we recall the interpolation inequality. 
	\begin{proposition}\label{Pro8}
		Let $1\leq  r,q,p\leq \infty$ be such that
		\[\frac{1}{r}  =\frac{\theta}{p} + \frac{1-\theta}{q} , \quad  \theta\in(0,1)\,,\]
		and let $u\in L^q(\RR^N)  \cap L^p(\RR^N)$. Then, we have $u\in L^r(\RR^N)$, and 
		\[\|u\|_{L^r(\RR^N)}  \leq  \|u\|^\theta_{L^p(\RR^N)}  \|u\|^{1-\theta}_{L^q(\RR^N)} \,.\]
		In particular, 	if  $u\in X$, then  for any $1<r<\infty$,  there holds true
		\begin{align*} 
			\|u\|_{L^r(\RR^N)} \leq \|u\|^\frac{1}{r}_{L^1(\RR^N)}   \|u\|^{1-\frac{1}{r}}_{L^\infty(\RR^N)}  \leq  \|u\|_{X} \,. 
		\end{align*}
	\end{proposition}

	\section{Regularizing equations}
	For any $\varepsilon>0$, we study a regularizing  equation to Eq \eqref{1}:  
	\begin{align}\label{P_epsilon}
		\left\{
		\begin{array}
			{ll}%
			\partial^\alpha_t u  -\varepsilon \Delta u  =\operatorname{div} \Theta(u)    &\text{in }  Q_T\,  ,
			\\
			u(x,0)=u_0(x) &\text{in }
			\RR^N \, .
			\\
		\end{array}
		\right.
	\end{align}
	Then, we look for the existence of mild solutions of Eq \eqref{P_epsilon} as a fixed point of the map 
	\[ 
	\mathcal{T} (u) (t) = Z_\varepsilon(t) * u_0 +  \int^t_0 \nabla Y_\varepsilon(t-\tau)  * \Theta (u)(\tau)  \, d\tau\,, 
	\]
	where $Z_\varepsilon, Y_\varepsilon$ are the fundamental solutions to Eq \eqref{1.5a}  with diffusion $\varepsilon \Delta  u$, modified from $Z, Y$  in  \eqref{1.6}.
	\\
	In fact, we show that  $\mathcal{T}$ is a  contraction mapping from $L^\infty\big(0,T; X\big)\to L^\infty\big(0,T; X\big)$ for some $T>0$.
	Then, we have the following theorem.
	\begin{theorem}\label{Themildsol}
		Let $\frac{1}{2}\leq s<1$,  $\alpha\in(0,1)$  $m\geq 1$, and   $u_0 \in X$.  For any $\varepsilon>0$, then Eq \eqref{P_epsilon} has a unique mild solution $u_\varepsilon \in L^\infty\big(0,T; X\big)$ for some $T>0$.
	\end{theorem}
	\begin{proof}[Proof of Theorem \ref{Themildsol}] The proof follows by way of the following lemma. 
		\begin{lemma}\label{le6} Assume hypotheses as in Theorem \ref{Themildsol}. Then, there exists a time $T_0\in (0,1)$, and a real number $\gamma>0$ such that   
			\begin{align}\label{3.112}
				\big\|\mathcal{T}(u)-\mathcal{T}(v)\big\|_{L^\infty\left(0,T_0; X\right)} \leq C M_0^{m}  T_0^{\gamma}\big\|u-v\big\|_{L^\infty\left(0,T_0; X\right)}  \,, \quad\forall  u,v\in  \overline{B(0,M_0)}\,, 
			\end{align}
			where $\overline{B(0,M_0)}$ is the closed ball in $L^\infty \left(0,T_0; X\right)$  with center at $0$, and  radius $M_0$; and  constant $C>0$ only depends on the parameters involved  such as $\varepsilon, N, m, \alpha, s$.
		\end{lemma}
		\begin{proof}[Proof of Lemma \ref{le6}]  Let us drop the dependence on $\varepsilon$ of $Z_\varepsilon$, and  $Y_\varepsilon$ for short. The reader should keep in mind that the estimates of $Y_\varepsilon(x,t)$ below depend on $\varepsilon$.
			\\
			By \eqref{1.10c}, applying Young's inequality yields
			\begin{align}\label{3.114}
				\big\|\mathcal{T}(u)-\mathcal{T}(v)  \big\|_{L^p(\RR^N)} &=  \left\| \int_{0}^{t} \nabla Y (t-\tau) * \big[\Theta(u)(\tau)-\Theta(v)(\tau) \big] \, d\tau \right\|_{L^p(\RR^N)} \nonumber
				\\
				& \leq  \int_{0}^{t}\big\|\nabla Y (t-\tau) \big\|_{L^r(\RR^N)}  \big\|\Theta(u)(\tau)-\Theta(v)(\tau)\big\|_{L^q(\RR^N)}  \, d\tau  \nonumber
				\\
				&\leq C \int_{0}^{t}(t-\tau)^{\frac{\alpha}{2}-1-\frac{\alpha N}{2}\left(1-\frac{1}{r}\right)}  \big\|\Theta(u)(\tau)-\Theta(v) (\tau) \big\|_{L^{q}(\RR^N)} \,
				d\tau\,,
			\end{align}
			with \, $\frac{1}{p} +1 =\frac{1}{r} +\frac{1}{q}$, 
			$r\in \big[1,\kappa_3(N)\big)$,  and $C=C(\varepsilon, N, \alpha, r)>0$. 
			\\
			It is of course that we will consider $p=1$ and $p=\infty$ alternatively in \eqref{3.114}, and  $r\in \big[1,\kappa_3(N)\big)$ will be chosen with respect to $p$. 
			\\
			At the moment, we want to estimate \eqref{3.114}. Observe that
			\begin{align*}
				\Theta(u)-\Theta(v)  &=(|u|^{m}-|v|^{m}) \nabla (-\Delta)^{-s}[u] +|v|^{m} \nabla (-\Delta)^{-s}[u-v]
				\\
				&=(|u|^{m}-|v|^{m}) \mathcal{R}\mathcal{I}_{2s-1}[u]+|v|^{m}\mathcal{R}\mathcal{I}_{2s-1}[u-v]  \,.
			\end{align*}
			By H\"older's inequality and the fact that $\mathcal{R}$ is bounded on $L^p(\RR^N)$, we obtain
			\begin{align*}
				&\big\|\Theta(u)(t)-\Theta(v) (t)\big\|_{L^{1}(\RR^N)}\nonumber 
				\\
				&\leq \big\|\big(|u|^{m}-|v|^{m}\big)(t)\big\|_{L^{2}(\RR^N)} \big\|\mathcal{R}\mathcal{I}_{2s-1}[u] (t)\big\|_{L^{2}(\RR^N)} +\big\||v|^{m}(t)\big\|_{L^{2}(\RR^N)} \big\|\mathcal{R}\mathcal{I}_{2s-1}[u-v](t)\big\|_{L^{2}(\RR^N)}
				\\
				&\lesssim \big\|\big(|u|^{m}-|v|^{m}\big)(t)\big\|_{L^{2}(\RR^N)} \big\|\mathcal{I}_{2s-1}[u] (t)\big\|_{L^{2}(\RR^N)} +\big\||v|^{m}(t)\big\|_{L^{2}(\RR^N)} \big\|\mathcal{I}_{2s-1}[u-v](t)\big\|_{L^{2}(\RR^N)}\,.
			\end{align*}
			With the last inequality noted, it follows from the Riesz potential estimate in \eqref{2.18}, and the interpolation inequality in Proposition \ref{Pro8} that
			\begin{align}\label{3.50}
				&\big\|\Theta(u)(t)-\Theta(v) (t)\big\|_{L^{1}(\RR^N)}\nonumber 
				\\
				&\lesssim  M_0^{m-1}\|(u-v)(t)\|_{L^{2}(\RR^N)}\left\|u(t)\right\|_{L^{\frac{2N}{N+2(2s-1)}}(\RR^N)}
				+ \|v(t)\|^m_{L^{2m}(\RR^N)}\left\|(u-v)(t)\right\|_{L^{\frac{2N}{N+2(2s-1)}}(\RR^N)} \nonumber
				\\
				&\lesssim M_0^{m}\|(u-v)(t)\|_{X}\,,\quad \forall u, v \in \overline{B(0,M_0)} \,.
			\end{align}
			Note that $\frac{2N}{N+2(2s-1)}>1$ since $N\geq 2$, and $1/2\leq s<1$.
			\\
			By choosing  $p=r=q=1$ in \eqref{3.114},  we obtain
			\begin{align}\label{3.777}
				\big\|\mathcal{T}(u)(t)-\mathcal{T}(v) (t)\big\|_{L^{1}(\RR^N)} & \leq C  M_0^{m} \int_{0}^{t}(t-\tau)^{-1+\frac{\alpha}{2}}\big\|(u-v)(\tau)\big\|_{X} \, d\tau    \nonumber
				\\
				& \leq C M_0^{m} T_0^{\frac{\alpha}{2}}\|u-v\|_{L^\infty\left(0, T_0; X\right)}\,,
			\end{align}
			for any $t \in(0, T_0)$, with $C=C(\varepsilon, N, \alpha,m, s)$.
			\\
			Next,  we set $q_s=\frac{Nq}{N+q(2s-1)}$. Note that $q_s>1 \Leftrightarrow q> \frac{N}{N+1-2s}$. 
			\\
			By a similar argument as in the proof of \eqref{3.50}, we have
			\begin{align}\label{3.51}
				&\big\|\Theta(u)(t)-\Theta(v)(t)\big\|_{L^{q}(\RR^N)} \nonumber\\
				&\leq \big\| \left(|u|^{m}-|v|^{m}\right)(t) \big\|_{L^{\infty}(\RR^N)}\big\|\mathcal{I}_{2s-1} [u](t)\big\|_{L^{q}(\RR^N)} +  \big\| |v|^{m}(t) \big\|_{L^{\infty}(\RR^N)}\big\|\mathcal{I}_{2s-1} [u-v](t)\big\|_{L^{q}(\RR^N)}\nonumber
				\\
				&\lesssim M_0^{m-1}\big\|(u-v)(t)\big\|_{L^{\infty}(\RR^N)} \big\|u(t)\big\|_{L^{q_s}(\RR^N)} + M_0^{m}\big\|(u-v)(t)\big\|_{L^{q_s}(\RR^N)}  \nonumber
				\\
				&\lesssim M_0^{m-1}\left(\|(u-v)(t)\|_{X}\|u(t)\|_{X}+\|v(t)\|_{X}\|(u-v)(t)\|_{X}\right) \nonumber
				\\
				&\leq C M_0^{m}\|(u-v)(t)\|_{X} \,.  
			\end{align}
			Pick $p=\infty$, and $q>N>\frac{N}{N+1-2s}$ in \eqref{3.114}.  Then,  observe that  $r=\frac{q}{q-1} 
			\in \big[1, \frac{N}{N-1}\big) \subset \big[1,  \kappa_3(N)\big)$. 
			\\
			Inserting \eqref{3.51} into \eqref{3.114} yields 
			\begin{align}\label{3.1000}
				\big\|\mathcal{T}(u)(t) -\mathcal{T}(v)(t)\big\|_{L^{\infty}(\RR^N)} &\leq  C M_0^{m} \int^t_0(t-\tau)^{\frac{\alpha}{2}-1-\frac{\alpha N}{2}\left(1-\frac{1}{r}\right)  }\|(u-v)(\tau)\|_{X}  \,d\tau  \nonumber  
				\\
				&\leq C M_0^{m}  T_0^{\frac{\alpha}{2}-\frac{\alpha N}{2q}}\|u-v\|_{L^\infty\left(0, T_0;X\right)}   \,,
			\end{align}
			for $t\in(0,T_0)$. 
			\\
			From \eqref{3.777} and \eqref{3.1000}, we obtain \eqref{3.112} with $\gamma=\frac{\alpha}{2}-\frac{\alpha N}{2q}$ for  $q>N$.   
			\\
			Therefore, we complete the proof of Lemma \ref{le6}.
		\end{proof}
		
		Now, it is enough to finish the proof of Theorem \ref{Themildsol}. By choosing $T_0>0$ small enough such that 
		\begin{equation}\label{3.222}
			C M_0^{m}  T_0^{\gamma}<1/2  \,,
		\end{equation} 
		then it is clear that $\mathcal{T}$ is a contraction mapping on $L^\infty\left(0,T_0; X\right)$.
		It remains to show that $\mathcal{T}$ maps $\overline{B(0,M_0)}$ to $\overline{B(0,M_0)}$, with 
		$M_0=2 \|u_0\|_{X}$. 
		\\
		Indeed, let us take $v = 0$ in \eqref{3.112}. Then,  for $u\in \overline{B(0,M_0)}$,  and by \eqref{3.222}, we have
		\begin{align}\label{3.11111}
			\|\mathcal{T}(u)\|_{L^\infty\left(0, T_0; X\right)}  &\leq \|\mathcal{T}(0)\|_{L^\infty\left(0, T_0; X\right)}+ C M_0^{m} T_0^\gamma\|u\|_{L^\infty\left(0, T_0; X\right)}  \nonumber
			\\
			&\leq \|\mathcal{T}(0)\|_{L^\infty\left(0, T_0; X\right)}+ \frac{1}{2} M_0\,.    
		\end{align}
		On the other hand, it is known that $Z(x,t)$ is a   probability density function. That is 
		\[
		\|Z(t)\|_{L^1(\RR^N)}=1 \,, \quad \forall  t>0\,.
		\]
		Thus, it follows from Young's inequality that 
		\begin{align*}
			\|\mathcal{T}(0)(t)\|_{L^1(\RR^N)} =   \|Z*u_0(t)\|_{L^1(\RR^N)} \leq \|Z(t)\|_{L^1(\RR^N)} \|u_0\|_{L^1(\RR^N)}   = \|u_0\|_{L^1(\RR^N)}  \,, \end{align*}
		and
		\begin{align*}
			\|\mathcal{T}(0)(t)\|_{L^\infty(\RR^N)}\leq  \|Z(t)\|_{L^1(\RR^N)} \|u_0\|_{L^\infty(\RR^N)}= \|u_0\|_{L^\infty(\RR^N)}\, \quad \text{for }  t>0\,.
		\end{align*}
		Hence,
		\begin{align}\label{3.52}
			\|\mathcal{T}(0)\|_{L^\infty(0,T_0;X)}\leq \|u_0\|_{X} \leq \frac{M_0}{2}\,.    
		\end{align}
		Inserting \eqref{3.52} into \eqref{3.11111} yields the conclusion. 
		\\
		This puts an end to proof of Theorem \ref{Themildsol}.
	\end{proof}
	
	Since the regularity of mild solutions of Eq \eqref{P_epsilon} is not sufficient for us to prove some a priori estimates,  then we have to study weak solutions of the homogeneous Dirichlet boundary condition corresponding to Eq \eqref{P_epsilon}:
	\begin{align}\label{Prob-R}
		\left\{\begin{array}{ll}%
			\partial^\alpha_t \tilde{u}  -\varepsilon \Delta \tilde{u}  =\operatorname{div}  \big( \Theta(u_\varepsilon) \mathbf{1}_{\Omega} \big)  & \text{ in }  \Omega \times(0,T)\,,
			\\
			\tilde{u}=0  \, &\text{ on }\partial \Omega \,,
			\\
			\tilde{u}(x,0) =u_0(x)  &\text{ in } \Omega\,,
		\end{array}
		\right.
	\end{align}
	where  $\Omega$ is a bounded domain in $\RR^N$, and  $\mathbf{1}_{\Omega}$ is denoted by the characteristic function on $\Omega$, and $u_\varepsilon$ is the unique mild solution of Eq \eqref{P_epsilon} corresponding to  $\operatorname{div}\big(\Theta(u) \mathbf{1}_{\Omega} \big)$ with initial data $u_0 \mathbf{1}_\Omega$  in Theorem \ref{Themildsol}. 
	\\
	Our strategy is as follows. First of all, for a given $\varepsilon>0$,  we prove an existence and uniqueness of weak bounded solution $\tilde{u}_\varepsilon$ to Eq \eqref{Prob-R}  in $\Omega\times(0,T_1)$, for some $T_1>0$. Next, we extend $\tilde{u}_\varepsilon(x,t)$ by zero on $  \big(\RR^N  \setminus \Omega\big) \times (0,T_1)$ (not labelled). Obviously,  we have $\tilde{u}_\varepsilon \in L^\infty\left( 0,T_1; X\right)$. Moreover, it is not difficult to verify that $\tilde{u}_{\varepsilon}$ is also a mild solution to  Eq \eqref{P_epsilon}. By the uniqueness result of mild solutions in class $L^\infty\big(0,T; X\big)$, we have that $\tilde{u}_\varepsilon= u_\varepsilon$ for a.e. $(x,t)\in Q_T$, where $T=\min\{T_0,T_1\}$. As a result, we can consider $u_\varepsilon$ as the unique weak solution of Eq \eqref{Prob-R}. This point of view enables us to pick some suitable test functions of $u_\varepsilon$-variable in order to establish some a priori estimates to $u_\varepsilon$. 
	\\
	
	Fix $\varepsilon>0$. For brief, we drop the dependence on $\varepsilon$ of $u_{\varepsilon}$, and $\tilde{u}_{\varepsilon}$.  By applying Theorem \ref{Themildsol} to $\operatorname{div} \big(\Theta(u) \mathbf{1}_{\Omega}\big)$ with initial data $u_0\mathbf{1}_{\Omega}(x)$, we find that there exists a unique mild solution $u\in L^\infty(0, T_0;X)$ to  Eq \eqref{P_epsilon}.
	\\
	Next, we denote
	\begin{align*}
		\mathcal{W}_{\Omega,T}=\big\{ u\in  L^2\big(0,T; H^1_0(\Omega)\big): g_{1-\alpha}*u\in C\big([0,T] ;L^2(\Omega)\big), (g_{1-\alpha}*u)|_{t=0}=0\big\} \,.   
	\end{align*} 
	We say that $\tilde{u}$ is a weak solution of Eq \eqref{Prob-R} if $\tilde{u}\in \mathcal{W}_{\Omega,T}$, and there holds true
	\begin{equation}\label{3.51a}
		\int^T_0 \int_{\Omega} \big(\varphi_t \big[g_{1-\alpha}*(\tilde{u}-u_0)\big]  -\varepsilon\nabla \tilde{u} \cdot \nabla \varphi \big) \,dx dt      =  \int^T_0 \int_{\Omega} \Theta(u) \cdot \nabla \varphi \, dx dt\, ,
	\end{equation}
	for all $\varphi\in L^2\big(0,T; H^1_0(\Omega)\big) \cap H^1\big(0,T; L^2 (\Omega)\big)$, and $\varphi(T)=0$.
	\\
	Now, we prove the existence and uniqueness of weak solutions to Eq \eqref{Prob-R}. The result is contained in the following theorem.
	\begin{theorem}\label{The-P_R} Assume hypotheses as in Theorem \ref{Themildsol}.  Then, Eq \eqref{Prob-R} has a unique bounded weak solution $\tilde{u}_{\varepsilon}\in \mathcal{W}_{\Omega, T}$, for some $T>0$. Moreover, for any $q\in[1,\infty]$, we have 
		\begin{align}\label{3.53}
			\|\tilde{u}_{\varepsilon}(t)\|_{L^q(\Omega)} \leq \|u_0\|_{L^q (\Omega)} \,,  \quad \text{for  } \, t\in (0,T)\,,
		\end{align}  
		and
		\begin{align}\label{3.54}
			\varepsilon\big\|\nabla \tilde{u}_{\varepsilon}\big\|^2_{L^2(\Omega\times (0,T))} \leq C\|u_0\|^2_{L^2(\Omega)} \,, 
		\end{align}  
		where $C=C(\alpha, T)$.
	\end{theorem}
	\begin{remark} \eqref{3.53} and \eqref{3.54} are the a priori estimates which play an important role in passing to the limit as $\varepsilon\to 0$ later on.  
	\end{remark}
	\begin{proof}[Proof of Theorem \ref{The-P_R}]
		First of all, we claim that $\operatorname{div} \big(\Theta(u) \mathbf{1}_{\Omega} \big)\in L^2\big(0,T; H^{-1}(\Omega)\big)$, $T>0$, where we denote $H^{-1}(\Omega)$ by the dual of $H^{1}_0(\Omega)$.
		\\
		Indeed, for any $t\in(0,T)$, we apply  H\"older's inequality  to obtain
		\begin{align}\label{3.40b}
			\big\|\operatorname{div} \big(\Theta(u) \mathbf{1}_{\Omega} \big)\big\|_{H^{-1}(\Omega)}  &
			= \sup_{ \|\psi\|_{H^1_0(\Omega)}\leq 1 }   \left|  \int_{\Omega}  |u|^{m}   \nabla (-\Delta)^{-s}u (x,t) \cdot  \nabla\psi (x) \,dx  \right|   \nonumber
			\\
			&\leq  \sup_{ \|\psi\|_{H^1_0(\Omega)}\leq 1 }   \|u\|^{m}_{L^\infty(Q_{T})} \|  \nabla (-\Delta)^{-s} u (t)  \|_{L^{2}(\Omega)}   \|\nabla \psi\|_{L^{2}(\Omega)}  \nonumber
			\\
			&\leq  \|u\|^{m}_{L^\infty(Q_{T})}  \| \nabla ( -\Delta)^{-s} u (t) \|_{L^{2}(\RR^N)} \,.
		\end{align}
		Next, applying the fundamental estimates of the Riesz transforms and the Riesz potential yields
		\begin{align}\label{3.40c}
			\|   \nabla (-\Delta)^{-s} u (t)\|_{L^{2}(\RR^N)}   &= \|\nabla \mathcal{I}_1 (-\Delta)^{1/2}  (-\Delta)^{-s} u (t)\|_{L^{2}(\RR^N)} = \| \mathcal{R}  (-\Delta)^{1/2-s} u (t) \|_{L^{2}(\RR^N)}  \nonumber
			\\
			& \lesssim  \|  \mathcal{I}_{2s-1}   u (t)\|_{L^{2}(\mathbb{R}^N)} \lesssim \|u (t)\|_{L^{\frac{2N}{N+2(2s-1)}}(\mathbb{R}^N)} \,.
		\end{align}
		Note that  $ 2_s=\frac{2N}{N+2(2s-1)}\in (1,2]$ since $N\geq 2$, and $s\in[1/2,1)$.
		\\
		Combining \eqref{3.40b} and \eqref{3.40c}	yields
		\begin{align}\label{3.4}
			\int^{T}_0\big\|\operatorname{div} \big(\Theta(u) \mathbf{1}_{\Omega} \big)(t)   \big\|^2_{H^{-1}(\Omega)}  \,dt & \lesssim       \|u\|^{2m}_{L^\infty(Q_{T})}  \int^{T}_0\|u (t)\|^{2}_{L^{ 2_s}(\RR^N)} dt \nonumber
			\\
			&\lesssim \|u\|^{2m+2-2_s}_{L^\infty(0,T;X)}   \|u\|^{2_s}_{L^{2_s}(Q_{T})} \,.
		\end{align}
		Thus, the claim follows.
		\\
		Now, we apply \cite[Theorem 3.1]{Zacher2009} to $f= \operatorname{div} \big(\Theta(u) \mathbf{1}_{\Omega} \big)$,  $\mathcal{V}=H^1_0(\Omega)$, and $\mathcal{H}=L^2(\Omega)$. Then, Eq \eqref{Prob-R}  has a unique weak solution  $\tilde{u}\in \mathcal{W}_{\Omega,T_1}$ for some $T_1>0$. 
		\\
		Set $T=\min\{ T_0, T_1 \}$. Remind that $T_0$ is the local existence time of $u$  in Theorem \ref{Themildsol}.
		\\
		Next, we show that  $\tilde{u}$ is bounded on  $\Omega\times(0,T)$. 
		To get the result, we are ready to apply \cite[Theorem 3.1]{Zacher2008} to $\tilde{u}$. 
		\\
		Observe that  $g_{1-\alpha}\in L^{p}(0,T_1)$ for some $p\in \left(1,\frac{1}{1-\alpha}\right)$. Thanks to \cite[Theorem 3.2]{Zacher2009}, we get $\tilde{u}\in L^{2p}(0,T_1;L^2(\Omega))$. Thus, it remains to verify that 
		\begin{equation}\label{3.5} 
			|\Theta(u)|^2  \in L^r\left( 0,T; L^q(\Omega) \right) , \end{equation}
		with 
		\begin{equation}\label{3.6}
			r\in \left[\frac{p'}{1-\alpha},\infty\right],\, q\in\left[ \frac{N}{2(1-\alpha)},\infty\right],\,  \text{ and }\,\frac{p'}{r} +  \frac{N}{2q}=1-\alpha  \,.
		\end{equation}
		For a straightforward computation, we can pick $r=\frac{2p'}{1-\alpha}$, and $q=\frac{N}{1-\alpha}$.  Since  $u\in L^\infty(0,T;X)$, then  for any $t\in(0,T)$  we get 
		\begin{align}\label{3.7a}
			\left\| \big|\Theta(u)(t)\big|^2 \right\|_{L^{q}(\Omega)} &\leq \|u\|^{2m}_{L^\infty(Q_{T})}  \left\|\mathcal{R}(-\Delta)^{1/2-s}u(t)\right\|^2_{L^{2q}(\RR^N)}
			\lesssim \|u\|_{L^\infty(Q_{T})}^{2m} \left\|\mathcal{I}_{2s-1} u(t)\right\|^2_{L^{2q}(\RR^N)}\nonumber
			\\
			&\lesssim\|u\|_{L^\infty(Q_{T})}^{2m}\|u(t)\|^2_{L^\frac{2qN}{N+2q(2s-1)}(\RR^N)} 
			\leq \big\|u\big\|_{ L^\infty(0,T; X)}^{2(m+1)} \,.
		\end{align}
		Note that $\frac{2qN}{N+2q(2s-1)}>1$ since $N\geq 2$, and $s\in [1/2,1)$.
		\\
		By integrating both sides of \eqref{3.7a} on $(0,T)$, we get
		\begin{align*}
			\int^{T}_0\left\|\big|\Theta(u)(t)\big|^2 \right\|_{L^{q}(\Omega)}^{r}  dt &\lesssim  T \big\|u\big\|_{ L^\infty(0,T; X)}^{2(m+1)r} \,,
		\end{align*} 
		which yields \eqref{3.5}.
		\\
		Therefore, there is a constant $C>0$ such that
		\[ 
		\|\tilde{u}(t)\|_{L^\infty(\Omega)} \leq C  , \quad\text{for   }t\in(0,T)\,,
		\]
		where $C$ depends on  $u_0, \Omega$, and the parameters involved.
		\\
		Thus,  $\tilde{u}$ is a bounded weak solution to Eq \eqref{Prob-R}. Next, we extend $\tilde{u}(x,t)$ by zero on $ \big(\RR^N  \setminus  \Omega\big) \times (0,T)$ (not labelled).  By the 
		above comments  (after \eqref{Prob-R}),  we have $\tilde{u}(x,t)  =  u(x,t)$  for a.e. $(x,t)\in Q_T$.  From here, we can consider $u$ as the unique weak solution of Eq \eqref{Prob-R}.
		\vspace{0.1in}
		\\
		\textbf{A priori estimates.}  Now, we prove \eqref{3.53} and \eqref{3.54}.
		Since kernel $g_{1-\alpha}$ is not smooth enough, then we have to work with $g_{1-\alpha,n}$ (a regularization of $g_{1-\alpha}$), introduced in Section \ref{Sec2}.  
		\\
		In what follows, we denote  $h_n=h_{\alpha,n}$, and  $g_n=g_{1-\alpha,n}$ for short. Then, we have the following lemma.
		\begin{lemma}\label{Lem3.1}
			$u  \in \mathcal{W}_{\Omega,T}$  is a weak solution of \eqref{P_epsilon}  if and only if for any test function $\psi(x) \in H^1_0(\Omega)$ and any $n \in  \mathbb{N}$, one has
			\begin{align}\label{3.2} 
				\int_{\Omega} \left(  \psi \frac{d}{dt}  [g_n  * (u-u_0)]     + \varepsilon   (h_n *  \nabla u ) \cdot\nabla \psi  \right) \, dx   =  -\int_{\Omega}  [h_n *  \Theta(u) ] \cdot  \nabla  \psi   \, dx\,,
			\end{align}
			for a.e.  $t\in (0,T)$.  
		\end{lemma}
		\begin{proof}[Proof of Lemma \ref{Lem3.1}]  The proof is just a slight modification of the one in \cite[Lemma 3.1]{Zacher2008}. Thus, we leave the details to the reader.
		\end{proof}	
		By a technical problem, we divide the proof of \eqref{3.53} into the  following cases:
		\\
		
		{\bf i) The case $q\geq 2$.}
		We apply \eqref{2.6} to $k=g_n$, and $G(u)=\displaystyle\frac{|u|^{q}}{q}$.  Note that $G$ is convex whenever $q>1$, and $g_n(t)$ is nonnegative and nonincreasing. 
		Thanks to Remark \ref{Rem2.1}, for a.e. $t\in(0,T)$  we obtain
		\begin{align}\label{3.7}
			G' \big(u ( t )\big) \frac{d}{dt}( g_n * u )( t ) \geq 
			\frac{d}{dt}\left[ g_n * G ( u )\right]  ( t )+
			\big[ -G \big(u ( t ) \big)
			+ G' \big(u ( t )\big)  u ( t ) \big] g_n ( t )  \,.
		\end{align}
		Therefore,
		\begin{align*}
			|u|^{q-2}u \frac{d}{dt}( g_n * u )( t ) \geq  \frac{1}{q}
			\frac{d}{dt} \left[g_n * |u|^q\right]  ( t )+
			\left(\frac{q-1}{q}\right) |u|^q(t) g_n ( t )  \,.
		\end{align*}
		Integrating both sides of  this inequality on $\Omega$   yields
		\begin{equation}\label{3.9}
			\int_{\Omega} |u|^{q-2}u \frac{d}{dt}( g_n * u )(x, t )  \, dx \geq  \frac{1}{q}\frac{d}{dt} \left[g_n * \|u\|^q_{L^q(\Omega)}\right]  ( t ) 
			+ \left(\frac{q-1}{q}\right) g_n ( t )\|u 
			(t)\|^q_{L^q(\Omega)}  \,.
		\end{equation}
		Next, let us take $\psi=  |u|^{q-2}u$, $q\geq 2$   as a test function to \eqref{3.2}. Then, we have 
		\begin{align*}
			&\int_{\Omega} \left( |u|^{q-2}u \frac{d}{dt}  [g_n  * (u-u_0)]+ \varepsilon   (h_n *  \nabla u ) \cdot\nabla |u|^{q-2}u  \right) (x,t) \, dx\\
			&=   -\int_{\Omega}  [h_n *  \Theta(u) ] \cdot  \nabla  |u|^{q-2}u (x,t) \, dx  \,, 
		\end{align*}
		and  
		\begin{align*}
			&\int_{\Omega}  |u|^{q-2}u \frac{d}{dt} (g_n  * u)  (x,t) \, dx  + \varepsilon   \int_{\Omega} (h_n *  \nabla u ) \cdot\nabla |u|^{q-2}u(x,t) \, dx   
			\\
			&=\int_{\Omega}  |u|^{q-2}u(x,t) \frac{d}{dt} [g_n (t) * u_0(x)]  \, dx  -\int_{\Omega}  [h_n *  \Theta(u) ] \cdot  \nabla  |u|^{q-2}u(x,t)  \, dx
			\\
			&=  g_n (t)  \int_{\Omega}  |u|^{q-2}u(x,t) u_0(x)  \, dx  -\int_{\Omega}  [h_n *  \Theta(u) ] \cdot  \nabla  |u|^{q-2}u(x,t)  \, dx \,.  
		\end{align*}
		With the last integral equation noted, it follows from \eqref{3.9} that
		\begin{align}\label{3.10}
			&  \frac{1}{q} \frac{d}{dt}\left[g_n * \|u\|^q_{L^q(\Omega)}\right]  ( t ) + \left(\frac{q-1}{q}\right)\|u\|^q_{L^q(\Omega)}  g_n ( t )   + \varepsilon \int_{\Omega}  (h_n *  \nabla u ) \cdot\nabla |u|^{q-2}u (x,t) \, dx \nonumber
			\\
			&\leq  g_n (t)  \int_{\Omega}  |u|^{q-2}u(x,t) u_0(x)  \, dx -\int_{\Omega}  [h_n *  \Theta(u) ] \cdot  \nabla  |u|^{q-2}u (x,t) \, dx\,.  
		\end{align}
		Applying H\"older's inequality and Young's inequality  yields 
		\begin{align*}
			\int_{\Omega}  |u|^{q-2}u(x,t) u_0(x)  \, dx    &\leq \| u (t)\|^{q-1}_{L^q(\Omega)} \| u_0  \|_{L^q(\Omega)}
			\leq \left(\frac{q-1}{q}\right) \| u  (t)\|^{q}_{L^q(\Omega)}  +  \frac{1}{q}  \| u_0  \|^q_{L^q(\Omega)}  \,.
		\end{align*}  
		From  the last inequality and \eqref{3.10}, we obtain
		\begin{align*}
			&\frac{1}{q}\frac{d}{dt} \left[g_n * \|u\|^q_{L^q(\Omega)}\right]  ( t ) + \varepsilon \int_{\Omega}  (h_n *  \nabla u ) \cdot\nabla |u|^{q-2}u(x,t)  \, dx \nonumber
			\\
			&\leq  g_n (t)  \frac{1}{q}  \| u_0  \|^q_{L^q(\Omega)}   -\int_{\Omega}  [h_n *  \Theta(u) ] \cdot  \nabla  |u|^{q-2}u(x,t)  \, dx  \,.  
		\end{align*}
		Thus, 
		\begin{align}\label{3.11}
			\frac{1}{q}\frac{d}{dt} \left[g_n *  \left( \|u\|^q_{L^q(\Omega)} -\| u_0  \|^q_{L^q(\Omega)}\right)\right] (t) &+ \varepsilon \int_{\Omega}  (h_n *  \nabla u ) \cdot\nabla |u|^{q-2}u(x,t)  \, dx\nonumber
			\\
			&\leq  -\int_{\Omega}  [h_n *  \Theta(u) ] \cdot  \nabla  |u|^{q-2}u(x,t)  \, dx   \,.  
		\end{align}
		By  \eqref{2.5} and \eqref{2.6a}, we can  pass to the limit as  $n\to\infty$  in \eqref{3.11} in order to obtain
		\begin{align}\label{3.13}
			F(t)  + \varepsilon  \int_{\Omega}   \nabla u  \cdot\nabla |u|^{q-2}u(x,t)  \, dx 
			+\int_{\Omega}    \Theta(u) \cdot  \nabla  |u|^{q-2}u(x,t)  \, dx  \leq 0 \,,
		\end{align}
		where 
		\begin{equation}\label{3.12}
			\frac{1}{q}\left( \|u(t)\|^q_{L^q(\Omega)} -\| u_0  \|^q_{L^q(\Omega)}\right) =g_\alpha* F(t)\,.
		\end{equation} 
		Observe that 
		\begin{align}\label{3.14}
			\int_{\Omega}  \Theta(u)  \cdot  \nabla |u|^{q-2}u(x,t)   \, dx 
			&=   (q-1) \int_{\Omega}    \nabla  (-\Delta)^{-s}u  \cdot    |u|^{m+  q-2} \nabla u  (x,t)  \, dx    		\nonumber
			\\
			& = \frac{q-1}{m+ q- 1}\int_{\Omega}   \nabla  (-\Delta)^{-s} u  \cdot  \nabla \left(|u|^{m+q-2}u \right) (x,t)  \, dx   \nonumber
			\\
			& =  \frac{q-1}{m+ q- 1}\int_{\Omega}  \big(|u|^{m+q-2}u \big)   (-\Delta)^{1-s} u    (x,t)  \, dx \,.
		\end{align}
		Thanks to the Stroock--Varopoulos inequality in Lemma \ref{LemSt-va},  we obtain
		\begin{align*}
			\int_{\Omega}  \big(|u|^{m+q-2}u \big)   (-\Delta)^{1-s} u    (x,t)  \, dx 
			\geq \frac{m+q-1}{\theta^2_q}  \big\|  |u|^{\theta_q-1}u  (t)   \big\|^2_{\dot{H}^{1-s}(\Omega)} \,,
		\end{align*}  
		with \,$\theta_q=\frac{m+q}{2}$.
		\\
		With this inequality note,  it follows from \eqref{3.14} that 
		\begin{equation}\label{3.14a}   
			\int_{\Omega}  \Theta(u)  \cdot  \nabla |u|^{q-2}u   (x,t)\, dx   \geq \frac{q-1}{ \theta_q^2 }\big\|  |u|^{\theta_q-1}u  (t)   \big\|^2_{\dot{H}^{1-s}(\Omega)} \,.
		\end{equation}
		In addition, it is easy to verify that 
		\begin{equation}\label{3.14b}
			\int_{\Omega}   \nabla u  \cdot  \nabla |u|^{q-2}u  (x,t)\, dx  =  (q-1)  \int_{\Omega} |u|^{q-2} |\nabla u|^2   (x,t)  \, dx   \geq 0\,.
		\end{equation}
		From \eqref{3.13}, \eqref{3.14a}, and \eqref{3.14b}, we find that  $F(t)\leq 0$ for a.e. $t\in(0,T)$. Since $g_\alpha$ is strictly positive  in $(0,\infty)$, then we obtain
		\[ \frac{1}{q} \left( \|u(t)\|^q_{L^q(\Omega)}- \| u_0  \|^q_{L^q(\Omega)}\right) =g_\alpha*F(t)\leq 0  \,, \]  
		for a.e. $t\in(0,T)$,
		which yields \eqref{3.53}. 
		\begin{remark}\label{Rem3.1}
			Observe that  $\|u(t)\|_{L^q(\Omega)}$ is nonincreasing on $(0,T)$  for $q\in[2,\infty)$. We emphasize that this conclusion is also true for the remaining cases. 
		\end{remark}
		\begin{remark}\label{Rem3.3}
			We deduce from \eqref{3.13} and  \eqref{3.14a} that 
			\begin{align}\label{3.14e}
				F(t) +\varepsilon (q-1)\int_{\Omega} |u|^{q-2} |\nabla u|^2 (x,t)\,dx+\frac{q-1}{ \theta_q^2 }\big\|  |u|^{\theta_q-1}u  (t)   \big\|^2_{\dot{H}^{1-s}(\Omega)}  \leq 0  \,,
			\end{align}
			for  a.e. $t\in(0,T)$. 
			\\
			By \eqref{3.12}, taking the convolution with  $g_{1-\alpha}*g_{\alpha}$ to both sides of \eqref{3.14e} provides us
			\begin{align}\label{3.14c}
				\frac{1}{q} \left[g_{1-\alpha}*  \left( \|u\|^q_{L^q(\Omega)} -\| u_0 \|^q_{L^q(\Omega)}\right)\right](t)  &+\varepsilon (q-1)\int^t_0 \int_{\Omega} |u|^{q-2} |\nabla u|^2 (x,t)\,dx d\tau  \nonumber
				\\
				&+  \frac{q-1}{ \theta_q^2 }  \int^t_0  \big\||u|^{\theta_q-1}u (\tau)\big\|^2_{\dot{H}^{1-s}(\Omega)}  d\tau\leq 0\,.
			\end{align}
			According to the monotonicity of $\|u(t)\|^q_{L^q(\Omega)}$ (see Remark \ref{Rem3.1}), one observes that
			\[
			\left[g_{1-\alpha}  *\|u\|^q_{L^q(\Omega)}\right]    (t) =  \int^t_0  \|u(\tau)\|^q_{L^q(\Omega)} \frac{(t-\tau)^{-\alpha}}{\Gamma(1-\alpha)}  \,d\tau \geq  C_\alpha  t^{1-\alpha} \|u(t)\|^q_{L^q(\Omega)}\,,   \]
			with  $C_\alpha=\frac{1}{(1-\alpha)\Gamma(1-\alpha)}$. 
			\\
			With the fact noted,  it follows then from \eqref{3.14c} that 
			\begin{align}\label{3.14d}
				C_\alpha t^{1-\alpha} \|u(t)\|^q_{L^q(\Omega)} &+\varepsilon q(q-1)\int^t_0 \int_{\Omega} |u|^{q-2} |\nabla u|^2 (x,t)\,dx d\tau \nonumber
				\\
				&+ \frac{q(q-1)}{\theta_q^2}  \int^t_0  \big\||u|^{\theta_q-1}u(\tau)\big\|^2_{\dot{H}^{1-s}(\Omega)}d\tau\leq C_\alpha t^{1-\alpha} \| u_0 \|^q_{L^q(\Omega)} \,.
			\end{align}
			As a result,  estimate \eqref{3.54} follows from \eqref{3.14e} when $q=2$.
			\\
			We emphasize that \eqref{3.14d} will be used to prove the decay estimates of the weak solution of Eq \eqref{1} later on. 
		\end{remark}
		
		{\bf ii) The case  $q=\infty$.}  Since \eqref{3.53} holds for any $q\geq 2$, then the conclusion  follows by passing to the limit as $q\to \infty$  in \eqref{3.53}. 
		\\
		
		{ \bf  iii) The case  $q=1$.}  To prove \eqref{3.53} for $q=1$,  we just mimic the proof of $L^q$-estimate  $q\geq 2$. In this case, we use  the test functions $G=S_\eta(u)$, and $\psi  = \chi_{\eta} (u)$, where   we denote  
		\[ \chi_{\eta} (r)  = \left\{\begin{array}{cc}
			{\rm sign} (r)  &   \text{  if }  |r|\geq\eta   \\
			\frac{r}{\eta} &  \text{  if }  |r|<\eta
		\end{array} \right.,\quad   S_\eta(r)  =  \int^r_0   \chi_{\eta} (\tau)\, d\tau  \]
		for any $\eta>0$.
		\\
		Obviously, $S_\eta(r)$ is a convex function.
		Then, we obtain as in \eqref{3.7} that
		\begin{align*}
			\int_{\Omega}  \chi_\eta(u)\frac{d}{dt}  (g_n *u) (x,t) 
			\,dx  & \geq 
			\frac{d}{dt}  \left[ g_n * \|S_\eta(u)\|_{L^1(\Omega)} \right](t) 
			\\
			&+  \int_{\Omega} \big(-S_\eta(u)+\chi_\eta(u) u \big)(x,t)  g_n(t)\, dx\,.
		\end{align*}
		A simple calculation shows that 
		\[ -S_\eta(u)+\chi_\eta(u)u =  \frac{u^2}{2\eta}  \mathbf{1}_{\{|u|<\eta\}}  +  \frac{\eta}{2} \mathbf{1}_{\{|u|\geq \eta\}}  \,.   \]
		With this identity noted, it follows from the last inequality that
		\begin{equation}\label{3.15}
			\int_{\Omega}  \chi_\eta(u)\frac{d}{dt}  (g_n *u) (x,t) 
			\, dx  \geq 
			\frac{d}{dt} \left[ g_n * \|S_\eta(u)\|_{L^1(\Omega)} \right](t) \,.
		\end{equation}
		Next, we pick  $\chi_\eta(u)$ as a  test function to \eqref{3.2}. Then, we obtain
		\begin{align*}
			\int_{\Omega}  \chi_\eta(u) \frac{d}{dt}  [g_n  * (u-u_0)] (x,t)\, dx    &+ \varepsilon  \int_{\Omega}  (h_n *  \nabla u ) \cdot  \nabla \chi_\eta(u)  (x,t)\,  dx 
			\\
			&=  -\int_{\Omega}  [h_n *  \Theta(u) ] \cdot  \nabla\chi_\eta(u)  (x,t) \, dx  \,.  
		\end{align*}
		Hence,
		\begin{align*}
			&\int_{\Omega}  \chi_\eta(u) \frac{d}{dt}  (g_n  * u) (x,t)\, dx    + \varepsilon  \int_{\Omega}  (h_n *  \nabla u ) \cdot  \nabla \chi_\eta(u)  (x,t)\,  dx 
			\\
			&= g_n(t) \int_{\Omega}  \chi_\eta(u (x,t))    u_0(x)\, dx -\int_{\Omega}  [h_n *  \Theta(u) ] \cdot  \nabla\chi_\eta(u)(x,t)   \, dx
			\\
			&\leq  g_n(t)  \|u_0\|_{L^1(\Omega)} -\int_{\Omega}  [h_n *  \Theta(u) ] \cdot  \nabla\chi_\eta(u)(x,t)   \, dx\,.  
		\end{align*}
		With the last inequality noted, it follows from \eqref{3.15} that
		\begin{align*}
			&\frac{d}{dt} \left[ g_n *  \left( \|S_\eta(u)\|_{L^1(\Omega)} 
			- \|S_\eta(u_0)\|_{L^1(\Omega)}\right)\right](t) + \varepsilon  \int_{\Omega}  (h_n *  \nabla u ) \cdot  \nabla \chi_\eta(u)  (x,t)\,  dx  
			\\
			&\leq g_n(t) \left( 
			\|u_0\|_{L^1(\Omega)} -\|S_\eta(u_0)\|_{L^1(\Omega)}  \right) -\int_{\Omega}  [h_n *  \Theta(u) ] \cdot  \nabla\chi_\eta(u)(x,t)   \, dx\,.
		\end{align*}
		At the moment, let us set 
		$$ \|S_\eta(u(t))\|_{L^1(\Omega)} 
		- \|S_\eta(u_0)\|_{L^1(\Omega)}=  g_\alpha *  F^{(\eta)}_1(t)\,.$$
		Thanks to \eqref{2.5} and\eqref{2.6a}, we can pass to the limit 
		as $n\to \infty$ in the last inequality in order to get
		\begin{align}\label{3.20}
			F^{(\eta)}_1(t)+ \varepsilon  \int_{\Omega}  \nabla u\cdot  \nabla \chi_\eta(u)  (x,t)\,  dx  &+\int_{\Omega}  \Theta(u)  \cdot  \nabla\chi_\eta(u) (x,t)  \, dx  \nonumber
			\\
			&\leq g_{1-\alpha}(t) \left( \|u_0\|_{L^1(\Omega)} -\|S_\eta(u_0)\|_{L^1(\Omega)}
			\right)\,.
		\end{align}
	Since $\chi_\eta (\cdot)$ is a nondecreasing function, then we have 
	$$ \int_{\Omega}    \nabla u  \cdot  \nabla \chi_\eta(u)  (x,t)\,  dx   =  \int_{\Omega}    |\nabla u|^2  \chi^\prime_\eta(u)  (x,t)\,  dx  \geq 0  \,.$$
	In addition, observe that  $\chi^\prime_\eta(l) = \eta^{m} \chi^\prime_{\eta^{m+1}}(|l|^m l)$, $l\in \RR$. Thus, we have
	\begin{align*}
		|u|^m  \nabla \chi_\eta   (u)  =  |u|^m \chi^\prime_\eta   (u) \nabla u &=   \frac{\eta^{m}}{m+1} \chi^\prime_{\eta^{m+1}}   (|u|^m 
		u)  \nabla   (|u|^m 
		u)\\   
		&= \frac{\eta^{m}}{m+1} \nabla\chi_{\eta^{m+1}}   (|u|^m u) \,. 
	\end{align*}
	With the identities noted,  we find that 
	\begin{align*}
		\int_{\Omega} \Theta(u) \cdot  \nabla\chi_\eta(u)  (x,t) \, dx  &=  \int_{\Omega} 
		\nabla (-\Delta)^{-s}  u  \cdot   |u|^m   \nabla  \chi_\eta(u)  (x,t) \, dx  
		\\
		&=  \frac{\eta^{m}}{m+1}\int_{\Omega} 
		\nabla (-\Delta)^{-s}  u  \cdot     \nabla\chi_{\eta^{m+1}}   (|u|^m u) (x,t) \, dx  
		\\
		&=   \frac{\eta^{m}}{m+1} \int_{\Omega} 
		\big[(-\Delta)^{1-s}  u  \big] \big[ \chi_{\eta^{m+1}}   (|u|^m u)\big] (x,t) \, dx  
		\geq 0,   
	\end{align*} 
	since 	the Stroock--Varopoulos inequality, see Lemma \ref{LemSt-va}.
	\\
	Thus, it follows from \eqref{3.20} that
	\[  F^{(\eta)}_1(t) \leq  g_{1-\alpha}(t)  \big( \|u_0\|_{L^1(\Omega)} -\|S_\eta(u_0)\|_{L^1(\Omega)} 
	\big) \,. \]
	Finally, passing to the limit as $\eta\to0$  yields
	\begin{align*}
		\|u(t)\|_{L^1(\Omega)}
		- \|u_0\|_{L^1(\Omega)}  &= \lim_{\eta\to 0} 
		g_{\alpha} *  F^{(\eta)}_1(t)
		\\
		& \leq \lim_{\eta\to 0}  g_\alpha* g_{1-\alpha}(t)  \big( \|u_0\|_{L^1(\Omega)} -\|S_\eta(u_0)\|_{L^1(\Omega)} 
		\big)
		\\
		& =\lim_{\eta\to 0} \big(\|u_0\|_{L^1(\Omega)} -\|S_\eta(u_0)\|_{L^1(\Omega)} \big)
		= 0,  
	\end{align*}  since 
	$g_\alpha* g_{1-\alpha}(t)=1$, and $\displaystyle\lim_{\eta\to 0} S_\eta(w)= w$ in $L^1(\Omega)$ for $w\in L^1(\Omega)$.
	\\
	Thus, we get \eqref{3.53} when $q=1$.
	\\
	
	{\bf iv) The case  $1<q<2$.} The proof is similar to the one of the case $q\geq 2$. But, one cannot pick $|u|^{q-2}u$ as a test function to Eq \eqref{Prob-R} as in \eqref{3.2} when $1<q<2$. Thus, we have to make a slight modification to this test function. Let us  introduce the following cut-off function:
	\\
	For any $\delta>0$, we set $$T_\delta(r)  = |r|^{q-2}r  \mathbf{1}_{\{|r|>\delta\}} + \delta^{q-2} (2r-\delta{\rm sign} (r))  \mathbf{1}_{\{ \frac{\delta}{2} < |r|\leq \delta\}}\,.$$ 
	It is clear that  function  $T_\delta$ is continuous 
	and nondecreasing, and  
	$$
	T^\prime_\delta(r) = 
	(q-1) |r|^{q-2}  \mathbf{1}_{\{|r|>\delta\}} + 2\delta^{q-2}   \mathbf{1}_{\{ \frac{\delta}{2} < |r|\leq \delta\}}  \,.$$
	Moreover, for $\delta>0$ fixed, one can verify easily that $T_\delta (u)\in L^2\big( (0,T) ; H^1_0(\Omega) \big)$ for  $u\in\mathcal{W}_{\Omega, T}$. Next, we set  
	$$
	G_\delta  (r)  =  \int^r_0  T_\delta (\tau) \, d\tau = \left(\frac{|r|^q -\delta^q }{q}+\frac{\delta^q}{4}\right)\mathbf{1}_{\{|r|>\delta\}}  +  \delta^{q-2}  \left(r^2  -  \delta |r|  +\frac{\delta^2}{4}\right) \mathbf{1}_{ \{\frac{\delta}{2} < |r|\leq \delta\}}     \,. 
	$$
	Obviously, $G_\delta\in \mathcal{C}^1(\RR)$ is a convex function. 
	\\
	Then, we repeat the proof of  the case  $q\geq 2$ with $G(u)=G_\delta(u)$, and $\psi=T_\delta(u)$. Indeed, proceeding  as in \eqref{3.7} yields
	\begin{align*}
		G^\prime_\delta(u) \frac{d}{dt}( g_n * u)( t ) \geq \frac{d}{dt}\left[ g_n *  G_\delta ( u )\right]  ( t )+\big[ - G_\delta \big(u ( t ) \big)+  G^\prime_\delta \big(u ( t )\big)  u ( t ) \big] g_n ( t )  \,.
	\end{align*}
	Observe that
	\begin{align*}
		- G_\delta \big(u\big)+ G^\prime_\delta \big(u\big)  u   =& \left[ \left(1-\frac{1}{q}\right) |u|^{q} 
		+\frac{\delta^{q}}{q}-\frac{\delta^q}{4}  \right]\mathbf{1}_{\{|u|>\delta\}}  
		+\delta^{q-2} \left( u^2_R -\frac{\delta^2}{4}  \right)\mathbf{1}_{\{\frac{\delta}{2}<|u|\leq \delta\}}  \,.
	\end{align*}
	With the identity noted, it follows from the indicated inequality that  
	\begin{align*}
		T_\delta(u) \frac{d}{dt}( g_n * u )( t ) \geq 
		\frac{d}{dt} \left[g_n *  G_\delta ( u )\right]  ( t )+  \left(1-\frac{1}{q}\right) 
		g_n(t) |u|^{q} \mathbf{1}_{\{|u|>\delta\}}(t) \,.
	\end{align*}
	Integrating the last inequality on $\Omega$  yields
	\begin{align}\label{3.21}
		\int_{\Omega} T_\delta(u) \frac{d}{dt}( g_n * u)( x,t ) \, dx & \geq \frac{d}{dt} \left[ g_n * \big\|  G_\delta (u)\big\|_{L^1(\Omega)}   \right](t)   \nonumber
		\\
		&+   \left(1-\frac{1}{q}\right) g_n(t) \int_{\Omega} \left|u
		\right|^{q} \mathbf{1}_{\{|u|>\delta\}}(x,t) \, dx \,.
	\end{align}
	By picking $\psi= T_\delta (u)$ as a test function to \eqref{3.2}, we get 
	\begin{align*}
		&\int_{\Omega} \left[ T_\delta (u) \frac{d}{dt}  \left[g_n  * \big(u - u_0\big)\right]    
		+ \varepsilon   (h_n *  \nabla u ) \cdot \nabla T_\delta (u) \right] (x,t)\, dx   
		\\
		&=  -\int_{\Omega}  [h_n *  \Theta(u) ] \cdot  \nabla T_\delta (u)(x,t) \, dx  \,.  
	\end{align*}
	Then,
	\begin{align*}
		&\int_{\Omega}  T_\delta (u) \frac{d}{dt} (g_n  * u )(x,t) \, dx  + \varepsilon   \int_{\Omega} (h_n *  \nabla u ) \cdot\nabla T_\delta (u)(x,t)\, dx   
		\\
		&=  g_n (t)  \int_{\Omega} T_\delta (u)(x,t) u_0(x)  \, dx  -\int_{\Omega}  [h_n *  \Theta(u) ] \cdot  \nabla  T_\delta (u) (x,t)\, dx \,.  
	\end{align*}
	Inserting  \eqref{3.21} into the last integral identity  yields
	\begin{align*}
		&\frac{d}{dt} \left[ g_n * \big\|  G_\delta ( u )\big\|_{L^1(\Omega)}   \right](t) 
		+ \left(1-\frac{1}{q}\right)  g_n(t) \big\| |u|^{q} \mathbf{1}_{\{|u|>\delta\}} (t)\big\|_{L^1(\Omega)}  
		\\
		&+ \varepsilon   \int_{\Omega} (h_n *  \nabla u ) \cdot\nabla T_\delta (u)(x,t)\, dx \nonumber
		\\
		&\leq  g_n (t)  \int_{\Omega} T_\delta (u)(x,t) u_0(x)  \, dx  -\int_{\Omega}  [h_n *  \Theta(u) ] \cdot  \nabla  T_\delta (u)(x,t)\, dx\,.
	\end{align*}
	We rewrite  the  inequality to obtain 
	\begin{align}\label{3.22}
		&\frac{d}{dt} \left[ g_n * \left(\big\|  G_\delta ( u )\big\|_{L^1(\Omega)}- \big\|  G_\delta ( u_0 )\big\|_{L^1(\Omega)}  \right) \right](t) 
		+ \left(1-\frac{1}{q}\right)  g_n(t) \big\| |u|^{q} \mathbf{1}_{\{|u|>\delta\}} (t)\big\|_{L^1(\Omega)}  \nonumber
		\\
		&+ \varepsilon   \int_{\Omega} (h_n *  \nabla u ) \cdot\nabla T_\delta (u)(x,t)\, dx + \int_{\Omega}  [h_n *  \Theta(u) ] \cdot  \nabla  T_\delta (u)  (x,t)\, dx  \nonumber
		\\
		&\leq  g_n (t)  \int_{\Omega} T_\delta (u)(x,t) u_0(x)  \, dx  - g_n (t) \big\|  G_\delta ( u_0 )\big\|_{L^1(\Omega)} \,.
	\end{align}
	Passing to the limit as $n\to \infty$ in \eqref{3.22}  yields
	\begin{align}\label{3.23}
		&F^{(\delta)}_2(t)  + \left(1-\frac{1}{q}\right)  g_{1-\alpha}(t) \big\| |u|^{q} \mathbf{1}_{\{|u|>\delta\}} (t)\big\|_{L^1(\Omega)}  + \varepsilon   \int_{\Omega}  \nabla u \cdot\nabla T_\delta (u)(x,t)\, dx \nonumber
		\\
		& + \int_{\Omega}   \Theta(u) \cdot  \nabla  T_\delta (u)(x,t)\, dx  \leq g_{1-\alpha}(t)  \int_{\Omega} T_\delta (u)(x,t) u_0(x)  \, dx  - g_{1-\alpha} (t)\big\|  G_\delta ( u_0 )\big\|_{L^1(\Omega)} \,,
	\end{align}
	with 
	\[  \big\|  G_\delta ( u (t))\big\|_{L^1(\Omega)}- \big\|  G_\delta ( u_0 )\big\|_{L^1(\Omega)} =g_\alpha* F^{(\delta)}_2(t) \,.    \]
	Again, observe that  $$\int_{\Omega}  \nabla u \cdot\nabla T_\delta (u)(x,t)\, dx,   \, \int_{\Omega}   \Theta(u) \cdot  \nabla  T_\delta (u)(x,t)\, dx\geq 0  $$  
	since $T^\prime_\delta(\cdot)\geq 0$, and the Stroock--Varopoulos inequality.
	\\
	Thus, it follows from \eqref{3.23} that
	\begin{align}\label{3.24}
		F^{(\delta)}_2(t)  &+ \left(1-\frac{1}{q}\right)  g_{1-\alpha}(t) \big\| |u|^{q}\mathbf{1}_{\{|u|>\delta\}} (t)\big\|_{L^1(\Omega)}  \nonumber
		\\
		&\leq g_{1-\alpha}(t)  \left(\big\| T_\delta (u(t)) u_0 \big\|_{L^1(\Omega)} - \big\|  G_\delta ( u_0 )\big\|_{L^1(\Omega)} \right).
	\end{align}
	On the other hand, we deduce from the interpolation inequality that 
	\begin{align*}
		\|u(t)\|_{L^q(\Omega)} \leq \|u(t)\|^{q/2}_{L^2(\Omega)}  \|u(t)\|^{1-q/2}_{L^1(\Omega)}  \leq \|u_0\|^{q/2}_{L^2(\Omega)}  \|u_0\|^{1-q/2}_{L^1(\Omega)}  ,\,\quad   1<q<2 \,.
	\end{align*}
	Therefore, $\|u(t)\|_{L^q(\Omega)}$ is uniformly bounded w.r.t $\varepsilon$, for all $t\in(0,T)$. In particular, $|u(x,t)|^q$ is integrable on $\Omega$. Since 
	$0\leq G_\delta (u)\leq \frac{1}{q}|u|^q$ for all $\delta>0$, and $G_\delta(u)\to \frac{1}{q}|u|^q$ as $\delta\to 0$, then it follows from the dominated convergence theorem that
	\[
	\big\|G_\delta ( u )\big\|_{L^1(\Omega)}\to \frac{1}{q} 
	\big\||u|^q\big\|_{L^1(\Omega)},  \,\, \big\|G_\delta ( u_0 )\big\|_{L^1(\Omega)}\to \frac{1}{q} 
	\big\||u_0|^q\big\|_{L^1(\Omega)}
	\]
	as $\delta\to 0$. As a result, we obtain
	\begin{align}\label{3.25}
		\lim_{\delta\to 0}  g_\alpha* F^{(\delta)}_2(t)   =\frac{1}{q} \left(\|u(t)\|^q_{L^q(\Omega)}-\|u_0\|^q_{L^q(\Omega)} \right)\,.
	\end{align}
	Next, observe that 
	$$ 
	\big|T_\delta (v)\big|\leq |v|^{q-1}\mathbf{1}_{\left\{ |v|>\delta  \right\}} + 
	\delta^{q-1}  \mathbf{1}_{\left\{ |v|\leq \delta  \right\}}\,,  \text{ and  } T_\delta(v)\to|v|^{q-2}  v
	$$
	as $\delta\to 0$. Thus, we deduce from the dominated convergence theorem and the H\"older inequality that
	\begin{equation}\label{3.26}   
		\big\| T_\delta (u(t)) u_0 \big\|_{L^1(\Omega)}  \to   \big\|  |u|^{q-2}   u(t) u_0 \big\|_{L^1(\Omega)}  
		\,. 
	\end{equation}
	From \eqref{3.25} and \eqref{3.26}, we can pass to the limit as $\delta\to 0$ in \eqref{3.24}. Then, 
	\begin{align}\label{3.27}
		\frac{1}{q}\left(\|u (t)\|^q_{L^q(\Omega)}- \|u_0 \|^q_{L^q(\Omega)} \right)  \leq & g_\alpha* \left[  g_{1-\alpha}(t) \left(\big\| |u|^{q-2}u(t)   u_0 \big\|_{L^1(\Omega)} - \frac{1}{q} \| u_0 \|^q_{L^q(\Omega)}\right.  \right.  \nonumber
		\\
		&\left. \left. -\left(1-\frac{1}{q}\right)\|u(t)\|^q_{L^q(\Omega)}\right)  \right]\,.
	\end{align}
	Note that we cannot apply the identity  $ g_\alpha*   g_{1-\alpha}(t)=1$ in \eqref{3.27} right now since the resulted function in the bracket (multiplying with $g_{1-\alpha}(t)$) is of $t$-variable. However, this is not 
	a matter. Indeed, applying H\"older's inequality and Young's inequality  yields
	\begin{align*}
		\big\||u|^{q-2}u(t) u_0 \big\|_{L^1(\Omega)}
		\leq  \| u(t)\|^{q-1}_{L^{q}(\Omega)}  \| u_0 \|_{L^q(\Omega)} \leq \left(1-\frac{1}{q}\right)\| u(t)\|^{q}_{L^{q}(\Omega)} +\frac{1}{q} \| u_0 \|^q_{L^q(\Omega)} \,.
	\end{align*}
	With the last inequality noted, we deduce from \eqref{3.27} that
	\[\|u (t)\|^q_{L^q(\Omega)}- \|u_0 \|^q_{L^q(\Omega)}  \leq   0\,.\]
	This puts an end to the proof of Theorem \ref{The-P_R}.
\end{proof}
\begin{remark}\label{Rem3.4}  
	Note that \eqref{3.14e}-\eqref{3.14d} also hold for $1<q<2$.      
\end{remark}
\begin{remark}
	For any $\varepsilon>0$ fixed, it follows from \eqref{3.53} that  
	\[
	\|  u_\varepsilon (t) \|_X  \leq \|u_0\|_X \, \quad  \text{for   }  t\in(0,T) \,. 
	\]
	If we repeat the above argument with initial datum $u_\varepsilon(T)$ instead of $u_0$, then we obtain the existence and uniqueness of weak solution $u_\varepsilon(t)$ (not labeled) to Eq \eqref{Prob-R}  on $[T, 2T]$. By a bootstrap argument, we obtain the global existence of   $u_\varepsilon(t)\in   \mathcal{W}_{\Omega, T}$ for any $T>0$. 
\end{remark}

\section{Proof of Theorem  \ref{MainThe1}} 
For any $R>0$, let us recall $B_R$  the open ball in $\RR^N$  with center at $0$, and radius $R$. Then, Eq \eqref{Prob-R} with $\Omega=B_R$ possesses a unique weak solution $u_{\varepsilon,R}\in \mathcal{W}_{B_R,T}$ for $T>0$. Moreover,  $u_{\varepsilon,R}$ satisfies \eqref{3.53}, \eqref{3.54}. 
It suffices to pass to the limit as  $R\to \infty$, and  $\varepsilon\to 0 $ alternatively.
To do that, we need a compactness result by Rakotoson--Teman  \cite{Rak}.
\begin{theorem}\label{Thecompactness}
	Let  $(V, \|.\|_V)$, $(H,\|.\|_H) $  be two separable Hilbert spaces. Assume that $V \subset H$ with a compact and dense embedding. Consider a sequence  $\{u_\delta\}_{\delta>0}$  converging weakly to a function $u$  in $L^2(0,T; V)$, $T < \infty$. Then, $u_\delta$  converges strongly to $u$ in  $L^2(0,T;H)$ if and only if
	
	$(i)$  $u_\delta(t)$ converges to $u(t)$ weakly in $H$ for a.e. $t\in(0,T)$ ,
	
	$(ii)$ $\displaystyle\lim_{|E|\to 0,  E\subset[0,T] } \sup_{\delta>0} \int_E \|u_\delta (t)\|^2_H  \, dt = 0$.
\end{theorem}
Set $v_{\varepsilon,R}=|u_{\varepsilon,R}|^{\theta_q-1} u_{\varepsilon,R}$  for $\varepsilon, R>0$.
Then, for $q>1$, it follows from \eqref{3.14d}, and Remark \ref{Rem3.4} that	
\begin{equation}\label{3.28}
	\int^T_0 \| v_{\varepsilon,R} \|^2_{\dot{H}^{1-s}(\Omega)}   \, dt \leq C  T^{1-\alpha}   \|u_0\|^q_{L^q(\RR^N)} \,, 
\end{equation} 
where $C=C(m, q, \alpha)>0$.
\\
By \eqref{3.28}, \eqref{3.53}, and the fact $\theta_q>1$, we find that for a given $M>0$, $v_{\varepsilon,R}$ is uniformly bounded in  $L^2\big(0,T; H^{1-s}(B_M)\big)$ w.r.t $\varepsilon, R>0$. 
It is of course that  we shall apply Theorem \ref{Thecompactness} to  sequence $\{v_{\varepsilon,R}\}$, $V=H^{1-s}(B_M)$, and $H=L^2(B_M)$.
\vspace{0.1in}
\\
{\bf Passing to the limit as $R\to\infty$.}
By the uniform bound of  $v_{\varepsilon,R}$ in  $L^2\big(0,T; V\big)$, there exists a subsequence of $\{v_{\varepsilon,R}\}$ (not labelled),  converging  weakly to $v_\varepsilon$ in $L^2(0,T; V)$ as $R\to\infty$.  Thus, it remains to show that $v_{\varepsilon, R}$ satisfies $(ii)$ in Theorem \ref{Thecompactness}.
\\
By  \eqref{3.53},   for any measurable set $E\subset [0,T]$, we have 
\begin{align*}
	\int_E \|v_{\varepsilon, R} (t)\|^2_{L^2(B_M)} \, dt 
	=  \int_E \big\||u_{\varepsilon, R}|^{\theta_q-1}  u_{\varepsilon, R} \big\|^2_{L^2(B_M)}  \, dt  
	\leq \int_E \|u_0 \|^{2\theta_q}_{L^{2\theta_q}   (\RR^N)}  \, dt  \leq |E|   \|u_0 \|^{2\theta_q}_{L^{2\theta_q} (\RR^N)}   \,,
\end{align*}
which implies that 
\[  \lim_{|E|\to 0} \sup_{R>0}  \int_E \|v_{\varepsilon, R} (t)\|^2_{L^2(B_M)}dt =0\,.  \]
As a result, $v_{\varepsilon, R}$ converges strongly to $v_\varepsilon$  in $L^2\big(B_M\times(0,T)\big)$ as $R\to\infty$. 
\\
Next, we claim that  $u_{\varepsilon,R}$ converges to $u_\varepsilon:=  \big|v_\varepsilon\big|^{\frac{1}{\theta_q}-1} v_\varepsilon$ in $L^1\big(B_M\times(0,T)\big)$.  
\\ 
Using the fundamental inequality 
\[\left||a|^{\alpha\beta-1}a-|b|^{\alpha\beta-1}b\right|\leq 2^{1-\beta}\left||a|^{\alpha-1}a-|b|^{\alpha-1}b\right|^\beta  \,,\]
for all $a,b\in\RR, \alpha>0$ and $\beta\in (0,1)$, and by H\"older's inequality, we get
\begin{align*}
	\int^T_0\int_{B_M}|u_{\varepsilon,R}-u_\varepsilon|\, dx\, dt
	&\lesssim \int^T_0\int_{B_M}\left||u_{\varepsilon,R}|^{\theta_q-1}u_{\varepsilon,R}-|u_\varepsilon|^{\theta_q-1}u_\varepsilon 
	\right|^{\frac{1}{\theta_q}}\, dx dt  \quad (\alpha=\theta_q, \, \beta=\frac{1}{\theta_q})
	\\
	&\lesssim  (T|B_M|)^{1-\frac{1}{2\theta_q}}\left(\int^T_0\int_{B_M}\left||u_{\varepsilon,R}|^{\theta_q-1}u_{\varepsilon,R}-|u_\varepsilon|^{\theta_q-1}u_\varepsilon \right|^2\, dx\, dt\right)^\frac{1}{2\theta_q}\\
	&=  (T|B_M|)^{1-\frac{1}{2\theta_q}}\|v_{\varepsilon,R}-v_\varepsilon\|^\frac{1}{\theta_q}_{L^2(B_M\times (0,T))}  \,.
\end{align*}
This yields the claim. 
\\
Furthermore, by  the $L^\infty$-bound of $u$ in  \eqref{3.53}, we deduce that
\begin{equation}\label{3.33}
	u_{\varepsilon,R}  \to u_\varepsilon \quad \text{in   }  L^p(B_M\times (0,T)) , \, \, 1\leq p<\infty\,.
\end{equation}
as $R\to\infty$.
\\
Since our argument is independent of $M$ and by \eqref{3.53} again, then we find that $u_{\varepsilon, R}$ converges to $u_\varepsilon$ in $L^p(Q_T)$  up to a subsequence.  
\\
To pass to the limit $R\to\infty$,  we write Eq \eqref{Prob-R} under the weak sense as in Definition \ref{Def1}.  That is 
\begin{align}\label{3.42}
	\int_{Q_T}  \left( \varphi_t  [g_{1-\alpha}*(u_{\varepsilon, R}-u_0)] -\varepsilon  \nabla u_{\varepsilon, R}  \cdot \nabla\varphi   \right)\,  dx dt   =
	\int_{Q_T}  |u_{\varepsilon, R}|^{m}  \nabla (-\Delta)^{-s} [u_{\varepsilon, R}] \cdot  \nabla   \varphi \, dxdt  \,,
\end{align}
for any test function $\varphi  \in H^1 \big(0,T; L^2(\RR^N)\big)  \cap   L^2\big(0,T; H^1(\RR^N)\big)$, $\varphi(T)=0$.
\\
Then, we first show that 
\begin{equation}\label{3.43}
	\lim_{R\to\infty}\int_{Q_T}  \left( |u_{\varepsilon, R}|^{m}  \nabla (-\Delta)^{-s} [u_{\varepsilon, R}] - |u_{\varepsilon}|^{m}\nabla (-\Delta)^{-s}[u_{\varepsilon} ]\right) \cdot  \nabla   \varphi \, dxdt = 0\,.
\end{equation}
Indeed, we rewrite   
\begin{align*}
	&\int_{Q_T}  \left( |u_{\varepsilon, R}|^{m}  \nabla (-\Delta)^{-s} [u_{\varepsilon, R}] - |u_{\varepsilon}|^{m}\nabla (-\Delta)^{-s}[u_{\varepsilon} ]\right) \cdot  \nabla   \varphi \, dx dt  \\
	&=	 \int_{Q_T}  |u_{\varepsilon, R}|^{m} \nabla (-\Delta)^{-s}  [u_{\varepsilon, R}-u_\varepsilon]\cdot\nabla\varphi  \,  dxdt 
	+  \int_{Q_T}  \big( |u_{\varepsilon, R}|^{m}  -|u_\varepsilon|^{m} \big) \nabla (-\Delta)^{-s} [u_\varepsilon] \cdot \nabla\varphi \, dxdt
	\\
	&:= A_1 + A_2  \,.
\end{align*}
To obtain  \eqref{3.43}, it suffices to show that $A_1, A_2\rightarrow  0 $ as $R\rightarrow  \infty$.
\\
For $A_1$, we show that   \begin{equation}\label{3.34}
	\nabla (-\Delta)^{-s}[u_{\varepsilon, R}]  \to  \nabla (-\Delta)^{-s} [u_\varepsilon] \, \text{ in }  L^2(Q_T)  \,.
\end{equation}
Indeed, it follows from \eqref{3.40c} and \eqref{3.53} that
\begin{align*}
	\big\| \nabla (-\Delta)^{-s} [u_{\varepsilon, R} - u_\varepsilon ] (t) \big\|^2_{L^2(\RR^N)} &\lesssim \big\|[u_{\varepsilon, R} - u_\varepsilon](t)\big\|^{2}_{L^{2_s} (\RR^N)}
	\\
	&\lesssim  \|u_0\|_{L^{2_s}(\RR^N)}^{2 -2_s}\big\|[u_{\varepsilon, R} - u_\varepsilon](t)\big\|^{2_s}_{L^{2_s}(\RR^N)} \,.
\end{align*}
Recall here $2_s=\frac{2N}{N+2(2s-1)}\in (1,2]$.
\\
By integrating both sides of the last inequality  on $(0,T)$, we obtain 
\begin{align}\label{3.35}
	\big\| \nabla (-\Delta)^{-s} [u_{\varepsilon, R}-u_\varepsilon]\big\|^2_{L^2(Q_T)} \lesssim \|u_0\|_{L^{2_s}(\RR^N)}^{2-2_s} \big\|u_{\varepsilon, R}-u_\varepsilon\big\|^{2_s}_{L^{2_s}(Q_T)} \,.
\end{align}
Letting $R\to \infty$ in \eqref{3.35}  yields \eqref{3.34}.
\\
Thanks to the H\"older inequality, \eqref{3.53},    and  \eqref{3.34},  we have
\begin{align*}
	|A_1| &\leq \|u_{\varepsilon,R}\|^{m}_{L^\infty(Q_T)}  \big\| \nabla (-\Delta)^{-s}[u_{\varepsilon, R} - u_\varepsilon]\big\|_{L^2(Q_T)} 
	\|\nabla\varphi\|_{L^2(Q_T)}   
	\\
	&\leq \|u_0\|^{m}_{L^\infty(\RR^N)} \big\| \nabla (-\Delta)^{-s}[u_{\varepsilon, R} - u_\varepsilon]\big\|_{L^2(Q_T)}
	\|\nabla\varphi\|_{L^2(Q_T)}\,.
\end{align*} 
Hence, one obtains $A_1\to 0$ as $R\to \infty$.
\\
Concerning $A_2$, observe that $u_{\varepsilon,R} \to u_\varepsilon$ for a.e. $(x,t)\in Q_T$  (up to a subsequence if necessary). In addition,  we deduce from \eqref{3.53} that    
\[  \left| \big( |u_{\varepsilon,R}|^{m}  -|u_\varepsilon|^{m} \big) \nabla (-\Delta)^{-s}[u_\varepsilon] \cdot \nabla\varphi  \right|  \leq 2\|u_0\|^{m}_{L^\infty(\RR^N)}  \left| \nabla (-\Delta)^{-s} [u_\varepsilon] \cdot \nabla\varphi\right| \,.\]
Note that  $\nabla (-\Delta)^{-s}[u_\varepsilon] \cdot \nabla\varphi$ is integrable on $Q_T$ according to \eqref{3.35}. By applying the dominated convergence theorem, we find that  $A_2\to 0$, as $R\to \infty$.
\\
In summary, we obtain \eqref{3.43}.
\\
Finally, observe that   the integral  
\[
\int_{Q_T} \nabla u_{\varepsilon,R}  \cdot  \nabla \varphi  \, dxdt  \to \int_{Q_T} \nabla u_\varepsilon  \cdot  \nabla \varphi  \, dxdt \]  as $R\to \infty$ since \eqref{3.54}.
\\
Hence, we obtain a global existence of weak solution $u_\varepsilon$ in $\RR^N\times(0,\infty)$ to Eq \eqref{P_epsilon}.
\begin{remark}
	Observe that $u_\varepsilon$ satisfies \eqref{1.10} for $q\in[1,\infty]$, and  satisfies  \eqref{1.11}, \eqref{3.14e}-\eqref{3.14d} for $q>1$. 
\end{remark}
\begin{remark}
	Note that $u_\varepsilon$ is a unique weak solution of Eq \eqref{P_epsilon}. Indeed, if $v_\varepsilon\in L^\infty\big(0,T; X\big)$ is another weak solution, then  $v_\varepsilon$  is also a mild solution to Eq \eqref{P_epsilon}. By the uniqueness of mild solutions in the class  $L^\infty\big(0,T; X\big)$, we deduce that $u_\varepsilon(x,t)=v_\varepsilon(x,t)$ for a.e. $(x,t)\in Q_T$, $T>0$. 
\end{remark}
\textbf{Passing to the limit as $\varepsilon\to 0$.} 
Since our argument above is independent of $\varepsilon$, then we just repeat the above proof for $u_\varepsilon$. Thus, we find that
there is a subsequence of $\{u_\varepsilon\}$ (not labeled) such that 
$u_\varepsilon\to u$ in $L^q(Q_T)$ for $1\leq q<\infty$.
\\
And, we also have as in \eqref{3.43} that 
\begin{equation}\label{3.431}
	\lim_{\varepsilon\to 0}	\int_{Q_T}  \left( |u_\varepsilon|^{m}  \nabla (-\Delta)^{-s} [u_\varepsilon] - |u|^{m}\nabla (-\Delta)^{-s} [u] \right) \cdot  \nabla   \varphi \, dx dt =  0
\end{equation}
for any test function $\varphi  \in H^1 \big(0,T; L^2(\RR^N)\big)  \cap   L^2\big(0,T; H^1(\RR^N)\big)$, $\varphi(T)=0$.
\\
Finally,  we can verify easily that 
$$
\varepsilon \int_{Q_T} \nabla u_\varepsilon  \cdot  \nabla \varphi  \, dxdt  \to  0
$$  as $\varepsilon\to 0$ since $\sqrt{\varepsilon} \nabla  u_\varepsilon$ is uniformly bounded in $L^2(Q_T)$  w.r.t  $\varepsilon>0$, see \eqref{3.54}.    \\
Hence, we obtain the existence of weak solution $u$ of Eq \eqref{1}.  
\\
In addition, one can verify easily that   \eqref{1.10} also holds true for $u$,  and  $u$  satisfies  \eqref{1.11}, \eqref{3.14c}, \eqref{3.14d} for all $q>1$.
\\
This puts an end to the proof of Theorem \ref{MainThe1}.
\begin{remark}\label{Rem4.2}
	From \eqref{3.40b}, \eqref{3.40c}, and \eqref{1.10}, for any $R>0$ we have 
	\begin{equation}\label{4.25}
		\sup_{t>0}	\| \operatorname{div} \Theta (u) (t) \|_{H^{-1}(B_R)}
		\lesssim \|u_0\|^{m+1}_{L^\infty(\RR^N)} \,.
	\end{equation} 
	Thus, we find that  $\partial_t^\alpha  u =  \operatorname{div} \Theta (u) \in L^\infty\big(0,T; H^{-1}(B_R)\big)$. 
	As a result, one has  $u\in\mathcal{C}^\alpha \big([0,T]; H^{-1}(B_R)\big)$. Therefore, $u$ possesses an initial trace $u_0$ in  $H^{-1}(B_R)$ for any $R>0$. 
\end{remark}


\section{Decay estimates of solutions}  
In this part, we study some decay estimates of the solution $u$. 
\begin{proof}[Proof of Theorem \ref{MainThe2}]  
	From \eqref{3.14d}, we obtain that 
	\begin{align}\label{4.2}
		\frac{q (q-1)}{ \theta_q^2 }  \int^t_\tau  \big\||u|^{\theta_q-1}u (\zeta)\big\|^2_{\dot{H}^{1-s}(\RR^N)}\,d\zeta
		\leq C_\alpha  (t-\tau)^{1-\alpha} \| u (\tau) \|^q_{L^q(\RR^N)}  
	\end{align}
	for $0<\tau<t<\infty$.
	\\
	Recall here $\theta_q =\frac{m+q}{2}$ for convenience.  By the Sobolev embedding, we have
	\begin{align}\label{4.3}
		\big\|u(t)\big\|^{\theta_q}_{L^{2^\star  \theta_q }(\RR^N)}  =  \big\||u|^{\theta_q-1}u  (t)\big\|_{L^{2^\star}(\RR^N)} \leq 
		C\big\||u|^{\theta_q-1}u (t)\big\|_{\dot{H}^{1-s}(\RR^N)}\,,
	\end{align}
	for $0<t<\infty$,
	with  $C=C(N,s)$, and  $2^\star=\frac{2N}{N-2(1-s)}$.
	\\
	Since $\|u(t)\|_{L^{q}(\RR^N)}$ is nonincreasing  w.r.t $t>0$, then 
	we deduce from \eqref{4.2} and \eqref{4.3} that
	\begin{align*}
		\frac{q(q-1)}{\theta^2_q} (t-\tau)   \|u(t)\|^{2\theta_q}_{L^{2^\star  
				\theta_q }(\RR^N)} \leq C(t-\tau)^{1-\alpha} \|u(\tau)\|^q_{L^q(\RR^N)}\,,
	\end{align*}
	for   $0<\tau<t<\infty$, where $C=C(\alpha, N, s)>0$.
	\\
	Hence,
	\begin{align}\label{4.4}
		\|u(t)\|^{2\theta_q}_{L^{2^\star \theta_q}(\RR^N)} \leq  C \frac{\theta^2_q}{q(q-1)}  (t-\tau)^{-\alpha} \|u(\tau)\|^q_{L^q(\RR^N)}  , \quad \text{for }  0<\tau<t<\infty\,.
	\end{align}
	Since we are going to use an iteration method starting from here, then it is important to note that the constant $C$ in \eqref{4.4} will not change step by step.
	\\
	Put
	\[
	t_n=t(1-2^{-n}), \quad q_{n+1}=  2^\star  \theta_n, \quad \theta_n=\frac{m+q_n}{2}, \,n\geq  0, 
	\] 
	where $q_0=q$,  $\theta_0=\theta_q$. 
	\\
	At the moment, we apply	\eqref{4.4}  to $t=t_{n+1}$, $\tau=t_n$, and  $q=q_n$. Then,
	\[
	\| u(t_{n+1}) \|^{ \lambda_0 q_{n+1} }_{L^{q_{n+1}}}    \leq C \frac{\theta^2_{q_n}}{q_n (q_n-1)  }  (t_{n+1}-t_n)^{-\alpha}   \| u(t_n)\|^{q_n}_{L^{q_n}}  \,  ,  \, \lambda_0 =\frac{N-2(1-s)}{N}  \,.
	\] 
	By induction, we get
	\begin{align}\label{4.6}
		\| u(t_{n+1}) \|_{L^{q_{n+1}} (\RR^N) } &\leq 
		\left[C\frac{\theta^2_{q_n}}{ q_n (q_n-1) }\right]^\frac{1}{\lambda_0 q_{n+1}}  
		\left[C\frac{\theta^2_{q_{n-1}}}{q_{n-1}(q_{n-1}-1)  }\right]^\frac{1}{\lambda^{2}_0 q_{n+1}}
		\dots \left[C\frac{\theta^2_{q_0}}{q_0(q_{0}-1)  }\right]^\frac{1}{\lambda^{n+1}_0 q_{n+1}}  \nonumber
		\\
		&\times  \left( t^{-\alpha} 2^{n+1} \right)^\frac{1}{\lambda_0 q_{n+1}}    \left( t^{-\alpha} 2^{n} \right)^\frac{1}{\lambda_0^{2} q_{n+1} } \dots \left( t^{-\alpha} 2 \right)^\frac{1}{\lambda_0^{n+1}  q_{n+1} }
		\times  \| u_0\|^\frac{q_0}{\lambda_0^{n+1}  q_{n+1}}_{L^{q_0}}\,.
	\end{align}
	It is not difficult to verify that
	\begin{equation}\label{4.6a}
		\lim_{n\rightarrow  \infty}  \lambda_0^{n+1} q_{n+1}  =  q_0+ \frac{m}{1-\lambda_0}.
	\end{equation}
	Next, by using \eqref{4.6a}, we obtain
	\begin{align}\label{4.7}
		\lim_{n\rightarrow  \infty} \left( t^{-\alpha} \right)^\frac{1}{\lambda_0 q_{n+1}}  \left( t^{-\alpha}  \right)^\frac{1}{\lambda_0^{2} q_{n+1} } \dots \left( t^{-\alpha} \right)^\frac{1}{\lambda_0^{n+1}  q_{n+1} }&= \lim_{n\rightarrow  \infty}  
		t^{ -\frac{\alpha}{\lambda_0 q_{n+1}} \displaystyle \sum^{n}_{j=0} \lambda_0^{-j}}  \nonumber
		\\
		&=  \lim_{n\rightarrow  \infty} t^{ -\frac{\alpha}{\lambda_0 q_{n+1}}  
			\frac{1-\lambda^{-(n+1)}_0 }{1-\lambda^{-1}_0 }} \nonumber
		\\
		&= t^{-\frac{\alpha}{q_0(1-\lambda_0)  +m}}  \,.
	\end{align}
	Similarly, we also have
	\begin{equation}\label{4.8}
		\lim_{n\rightarrow  \infty} C^\frac{1}{\lambda_0 q_{n+1}}   C^\frac{1}{\lambda_0^{2} q_{n+1} } \dots C^\frac{1}{\lambda^{n+1}_0 q_{n+1} } = C^{\frac{1}{q_0(1-\lambda_0)  + m}}\,,
	\end{equation}
	and 
	\begin{align}\label{4.9}
		\lim_{n\rightarrow  \infty} (2^{n+1})^\frac{1}{\lambda_0  q_{n+1}}   (2^n)^\frac{1}{\lambda^{2}_0 q_{n+1} }\dots 2^\frac{1}{\lambda^{n+1}_0 q_{n+1} } =  2^{ \frac{1} { (1-\lambda_0)\left(q_0(1-\lambda_0) + m\right)} }\,. 
	\end{align}
	Next, we put  
	\[
	Z_n  = q_n^\frac{1}{\lambda_0 q_{n+1}}   q_{n-1}^\frac{1}{\lambda^{2}_0 q_{n+1} } \dots q_0^\frac{1}{\lambda^{n+1}_0 q_{n+1} }\,.
	\]
	Let us show that $Z_n$ is convergent as $n\rightarrow \infty$. Indeed, we consider a power series
	\begin{align}\label{4.10}
		S_n (s) = s^n \ln q_n  +   s^{n-1} \ln q_{n-1} +\dots + s^{1} \ln q_{1} + \ln q_0 \,.
	\end{align}
	Obviously, the radius of convergence of $S_n(s)$ is $1$. Thus, $S_n (\lambda_0) $  converges absolutely to a real number $S_0$ as $n\rightarrow  \infty$.  
	On the other hand, we note that
	\[\lambda^{n+1}_0 q_{n+1} \ln Z_n = S_n (\lambda_0)  \,.\]
	Then,
	\begin{equation}\label{4.11}		\lim_{n\rightarrow\infty} Z_n  =\exp  \left\{\frac{S_0}{q_0+\frac{m}{1-\lambda_0}} \right\}\,.\end{equation}
	By an analogue as in the proof of \eqref{4.11},  there are  two real numbers $S_1, S_2>0$ such that 
	\begin{equation}\label{4.12}
		\lim_{n\rightarrow\infty}  (q_n-1)^\frac{1}{\lambda_0 q_{n+1}}  (q_{n-1}-1)^\frac{1}{\lambda^{2}_0 q_{n+1} } \dots (q_0-1)^\frac{1}{\lambda^{n+1}_0 q_{n+1}}  = S_1  \,,
	\end{equation}
	and
	\begin{equation}\label{4.14}
		\lim_{n\rightarrow\infty}  \theta_{q_n}^\frac{2}{\lambda_0 q_{n+1}}   \theta_{q_{n-1}}^\frac{2}{\lambda^{2}_0 q_{n+1}} \dots \theta_{q_0}^\frac{2}{\lambda^{n+1}_0 q_{n+1}}  = S_2 \,.
	\end{equation}		
	By inserting \eqref{4.7}, \eqref{4.8}, \eqref{4.9}, \eqref{4.11}, \eqref{4.12}, and \eqref{4.14} into \eqref{4.6}, we find that  
	\begin{equation*}
		\|u(t)\|_{L^\infty(\RR^N)} \leq C t^{-\frac{\alpha}{q(1-\lambda_0)+m}} \| u_0 \|^{\frac{q(1-\lambda_0)}{ q (1-\lambda_0)+ m }}_{L^{q}(\RR^N)} \,,
	\end{equation*}
	with $C=C(\alpha,s,N, m, q)>0$.
	This yields \eqref{1.5}.
	\\
	Thus, we get the proof of Theorem \ref{MainThe2}.
\end{proof}
Finally, we prove Theorem \ref{MainThe3}.
\begin{proof}[Proof of Theorem \ref{MainThe3}]  
	In the beginning, we can assume that $u_0 \in \mathcal{C}^\infty_c(\RR^N)$. As soon as we establish \eqref{1.7a} for such initial data, then we can utilize \eqref{1.7a} as a priori estimate in order to get the existence of the weak solution with initial data $u_0$  in $ L^1(\RR^N) \cap L^{q_0}(\RR^N)$, $q_0>1$.
	\\
	Now, we prove the $L^1-L^p$ decay estimate of weak solutions of Eq \eqref{1} for $1<p<\infty$  with regular initial data. 
	It follows from the interpolation inequality, and \eqref{1.10} that 
	\begin{align*}
		\| u(t) \|_{L^{p}(\RR^N)}  & \leq \| u(t) \|^{1-\frac{1}{p}}_{L^{\infty}(\RR^N)} \| u(t) \|^{\frac{1}{p}}_{L^{1}(\RR^N)}   \leq \| u(t) \|^{1-\frac{1}{p}}_{L^{\infty}(\RR^N)} \| u_0 \|^{\frac{1}{p}}_{L^{1}(\RR^N)}  \,.
	\end{align*}
	Next, applying \eqref{1.5} in Theorem \ref{MainThe2} to $q=q_0$ yields
	\[\| u(t) \|_{L^{\infty}(\RR^N)}  \lesssim   (t-t/2)^{-\frac{\alpha}{q_0(1-\lambda_0)+m}} \|u(t/2)\|^{\frac{q_0(1-\lambda_0)}{ q_0 (1-\lambda_0)+ m }}_{L^{q_0}(\RR^N)} \lesssim t^{-\frac{\alpha}{q_0(1-\lambda_0)+m}} \|u_0\|^{\frac{q_0(1-\lambda_0)}{ q_0 (1-\lambda_0)+ m }}_{L^{q_0}(\RR^N)} \,.   \]
	By combining the two indicated inequalities, we obtain 
	\begin{equation}\label{4.16}
		\| u(t) \|_{L^{p}(\RR^N)}  \lesssim  t^{-\frac{\alpha\left(1-\frac{1}{p}\right)}{q_0(1-\lambda_0)+m}} 
		\|u_0\|^{\frac{q_0(1-\lambda_0) \left(1-\frac{1}{p}\right)}{ q_0 (1-\lambda_0)+ m }}_{L^{q_0}(\RR^N)}
		\| u_0 \|^{\frac{1}{p}}_{L^{1}(\RR^N)} \,.
	\end{equation}
	
	It remains to prove the existence of weak solutions to Eq \eqref{1}  for $u_0\in  L^1(\RR^N)\cap L^{q_0}(\RR^N)$. Indeed, let   $u_{0,\delta}  =  u_0  *\varrho_{\delta}$, $\delta>0$, where $\{\varrho_\delta\}$  is the standard sequence of mollifier functions. Obviously, we have  
	$u_{0,\delta}\in  X$,   $\|u_{0,\delta}\|_{L^q(\RR^N)}\leq 
	\|u_{0}\|_{L^q(\RR^N)}$  
	for all $\delta>0$, and $u_{0,\delta}\to u_0$  in $L^q(\RR^N)$,  $1\leq q\leq q_0$ as $\delta\to 0$.
	\\
	Next, for any  $\delta>0$, let $u_\delta$ be the  solution of Eq \eqref{1}, constructed as in Theorem \ref{MainThe1} with initial datum $u_{0,\delta}$ .  
	\\
	Then, applying \eqref{1.5} to  $u_\delta$ yields 
	\[ \|u_\delta(t)\|_{L^\infty(\RR^N)} \leq C t^{-\frac{\alpha}{q_0(1-\lambda_0)+m}} \| u_0 \|^{\frac{q_0(1-\lambda_0)}{q_0 (1-\lambda_0)+ m }}_{L^{q_0}(\RR^N)}  ,\quad t\in(0,\infty)  \,. \]
	This implies that for any $\tau>0$ fixed, $\|u_\delta(\cdot,t)\|_X$ is uniformly bounded in $\RR^N\times(\tau,\infty)$  w.r.t  $\delta>0$.
	Since our proof of the existence in Theorem \ref{MainThe1} is local in time, then we can mimic it in order to obtain
	$u_\delta\to u$  in $L^p(\RR^N\times[\tau, T])$  $1\leq p<\infty$   up to a subsequence for $0<\tau<T<\infty$.
	\\
	As a result, one has $u_\delta(x,t)\to u(x,t)$ for a.e. $(x,t)\in \RR^N\times (0, \infty)$ by the diagonal argument up to a subsequence (not labeled). Furthermore, it is not difficult to verify that  
	\[ 
	\lim_{\delta\to 0}	\int_{Q_T}  \left( |u_\delta|^{m}  \nabla (-\Delta)^{-s} u_\delta - |u|^{m}\nabla (-\Delta)^{-s}u \right) \cdot  \nabla   \varphi \, dxdt =  0, \quad \forall \varphi\in\mathcal{C}^\infty_c(Q_T)\,.
	\]
	Thus, we conclude that $u$ is a weak solution of Eq \eqref{1}. Furthermore, $u$ also satisfies \eqref{1.10} for $q\in [1,q_0]$.
	\\
	This puts an end to the proof of Theorem \ref{MainThe3}.
\end{proof}
\begin{remark}\label{Rem5.4} We can obtain another the $L^{q_0}$-estimate by using the comparison result by the authors in \cite[Lemma 2.6]{Zac3}. 
	\\
	Indeed, for any $q>q_0$, we deduce from the interpolation inequality  in Proposition \ref{Pro8}, and \eqref{1.10} that
	\begin{align*}
		\| u(t) \|_{L^{q_0}(\RR^N)} \leq \|u(t)\|^{\sigma}_{L^{1}(\RR^N)} \| u(t) \|^{
			1-\sigma}_{L^{q}(\RR^N)}\leq \|u_0\|^{\sigma}_{L^{1}(\RR^N)} \| u(t) \|^{
			1-\sigma}_{L^{q}(\RR^N)} \,,
	\end{align*}
	with $\sigma= \frac{\frac{1}{q_0}  -\frac{1}{q}}{1-\frac{1}{q}} \in (0,1)$.
	\\
	Therefore,
	\begin{equation*}\label{4.20}
		\| u(t) \|_{L^{q}(\RR^N)} \geq  \|u_0\|^{-\frac{\sigma}{1-\sigma}}_{L^{1}(\RR^N)} \| u(t) \|^{\frac{1}{1-\sigma}}_{L^{q_0}(\RR^N)}  \,.
	\end{equation*}
	With the last inequality noted,  it follows from \eqref{3.14c}, and  \eqref{4.3} with $q=2^\star \theta_{q_0}$  that
	\begin{align}\label{4.15}
		& g_{1-\alpha}*\left( \|u(t)\|^{q_0}_{L^{q_0}(\RR^N)} -\| u_0 \|^{q_0}_{L^{q_0}(\RR^N)}\right)   +  C \left(\|u_0\|^{-\frac{\sigma}{1-\sigma}}_{L^{1}(\RR^N)} \| u(t) \|^{\frac{1}{1-\sigma}}_{L^{q_0}(\RR^N)}\right)^{2\theta_{q_0}}    \nonumber
		\\
		& \leq  g_{1-\alpha}*\left( \|u(t)\|^{q_0}_{L^{q_0}(\RR^N)} -\| u_0 \|^{q_0}_{L^{q_0}(\RR^N)}\right)+  C  \left\|u(t)\right\|^{2\theta_{q_0}}_{L^{q}(\RR^N)}  \leq 0\,,
	\end{align}
	where $C=C(q_0,N,m,s)>0$.
	\\
	Then, it suffices to apply \cite[Theorem 7.1]{Zac3}  to  $\|u(t)\|^{q_0}_{L^{q_0}(\RR^N)}$ with 
	\[ 
	\gamma = \frac{2\theta_{q_0}}{q_0(1-\sigma)}= \frac{m+q_0-\lambda_0}{q_0-1}  , \,  \text{ and }  \nu=  C \|u_0\|^{-\frac{2 
			\sigma \theta_{q_0}}{1-\sigma}}_{L^{1}(\RR^N)} \,.
	\]
	Thus,  we get
	\begin{equation}\label{5.21}
		\|u(t)\|_{L^{q_0}(\RR^N)}  \leq  \left(\frac{C'}{1+t^\frac{\alpha}{\gamma}} \right)^\frac{1}{q_0}  \lesssim  t^{-\frac{\alpha}{q_0\gamma}} ,\quad \text{for } t>0   \,, 
	\end{equation} 
	where $C'$ depends on $\|u_0\|_{L^{1}(\RR^N)}$,  $\|u_0\|_{L^{q_0}(\RR^N)} $, and the parameters involved.
	\\
	A straightforward computation shows that
	\[
	C' =   \|u_0\|^{q_0}_{L^{q_0}(\RR^N)}   t^{\alpha/\gamma}_0  ,\quad t^\alpha_0 = C''(\alpha,\gamma)  \frac{\|u_0\|^{{q_0}(1-\gamma)}_{L^{q_0}(\RR^N)}}{\nu} \,.
	\]
	It is clear that  the $L^{q_0}$-estimate in \eqref{1.7a} is better than its version in \eqref{5.21} when $t$ is large since $\frac{\alpha\left(1-\frac{1}{p}\right)}{q_0(1-\lambda_0)+m}>\frac{\alpha}{\gamma q_0}$.
\end{remark}
\begin{acknowledgement}   There is no conflict interest. 
	This research is funded by University of Economics Ho Chi Minh City (UEH), Vietnam.
\end{acknowledgement}


\end{document}